\documentclass[a4paper,11pt,oneside,reqno]{amsart}
\usepackage{amsthm,amsmath,amssymb}
\usepackage{hyperref}
\usepackage{orcidlink}
\usepackage{graphicx}
\usepackage{booktabs}
\theoremstyle{definition}
\newtheorem{definition}{Definition}
\newtheorem{example}[definition]{Example}
\newtheorem*{exampleprime}{Example~\ref*{ex:n25}\({}'\)}
\usepackage{etoolbox}
\AtEndEnvironment{example}{\hfill\(\triangleleft\)}
\AtEndEnvironment{exampleprime}{\hfill\(\triangleleft\)}
\AtEndEnvironment{remark}{\ifendmark\hfill\(\triangleright\)\fi\global\endmarktrue}
\AtEndEnvironment{remark*}{\ifendmark\hfill\(\triangleright\)\fi\global\endmarktrue}
\newif\ifendmark \endmarktrue
\newcommand{\remarkendhere}{\hfill\(\triangleright\)\global\endmarkfalse}
\theoremstyle{remark}
\newtheorem{remark}[definition]{Remark}
\newtheorem*{remark*}{Remark}
\newtheorem{openproblem}[definition]{Open problem}
\theoremstyle{plain}
\newtheorem{theorem}[definition]{Theorem}
\newtheorem{lemma}[definition]{Lemma}
\newtheorem{proposition}[definition]{Proposition}
\newtheorem{corollary}[definition]{Corollary}

\newtheorem*{thmmainu}{Theorem~\ref*{thm:main}\({}'\)}

\newcommand{\NN}{\mathbb{N}}
\newcommand{\ZZ}{\mathbb{Z}}
\newcommand{\QQ}{\mathbb{Q}}
\newcommand{\lpf}{\operatorname{lpf}}
\newcommand{\Om}{\Omega}
\newcommand{\Gam}{\Gamma}
\newcommand{\seqnum}[1]{\href{https://oeis.org/#1}{\underline{#1}}}

\begin{document}

\title[Truncations revisited]{Truncations of the ring of number-theoretic
  functions, revisited}
\author{Jan Snellman \orcidlink{0009-0002-6676-5068}}
\address{Matematiska Institutionen, Linköpings Universitet, 581 83 Linköping, Sweden}
\email{jan.snellman@liu.se}
% Added 2026-08-30 for the arXiv/report submission.  Primary 13D02 rather than 11N25:
% the object studied is a minimal free resolution, and the sieve function is what its
% invariants turn out to equal -- the result is about the resolution, stated in the
% language of the sieve, not a contribution to sieve theory.
% amsart re-prints \address/\email as a standalone block at the very end via
% \AtEndDocument; suppressed here in favour of showing it once under the title,
% matching the house style set in lattice_line_covers_preprint.tex.
\makeatletter
\let\addresses\@empty
\makeatother
\maketitle

\begin{center}
  \small
  Matematiska Institutionen, Linköpings Universitet, 581\,83 Linköping, Sweden\\
  \texttt{jan.snellman@liu.se}
\end{center}

\begin{abstract}
  Let \(K\) be a field containing \(\QQ\), let \(\Gam\) be the ring of all
  functions \(\NN^+\to K\) under Dirichlet convolution, and let \(\Gam_n\) be
  its truncation to functions supported on \([1,n]\). In \cite{Snellman2000} it was shown that \(\Gam_n\) is a
  polynomial ring modulo a monomial ideal \(I_n\) which is stable after
  reversing the order of the variables, and the Poincaré--Betti series of
  \(\Gam_n\) was computed in terms of the numbers \(C_{n,v}\) of minimal
  generators of \(I_n\) of least support \(v\).

  We prove that \(C_{n,v} = \Phi(n,p_v)\), Legendre's sifting function: the
  number of integers in \([1,n]\) free of prime factors \(\le p_v\). This
  identifies an invariant of a minimal free resolution with a classical object
  of sieve theory, and it refines to the grading by total degree, with the
  number of prime factors carried in a variable. As consequences we obtain: a proof of Conjecture 4.6 of
  \cite{Snellman2000}, which was left open there; the asymptotics
  \(C_n \sim \pi(n)^2/2\) of the total number of minimal generators, showing that
  the lower bound \(C_n \ge \binom{\pi(n)+1}{2}\) of \cite{Snellman2000} is
  asymptotically sharp, together with the exact order
  \(C_n-\binom{\pi(n)+1}{2} \sim \tfrac{16}{3}n^{3/2}/\log^{3}n\) of the error;
  and the identification of \(C_n\) with the OEIS sequence \seqnum{A182843}. We also record errata for \cite{Snellman2000}: one
  stated result is false, one worked example is miscomputed, and two proofs are
  incomplete. Corrected statements
  and complete proofs are given.
\end{abstract}

\section{Introduction}

Cashwell and Everett \cite{CashwellEverett} studied the ring
\begin{equation}
  \Gam = \{\,f \colon \NN^+\to K\,\}
\end{equation}
of \emph{number-theoretic functions}, with pointwise addition and the Dirichlet
convolution
\begin{equation}
  fg(n) = \sum_{ab = n}f(a)g(b);
\end{equation}
here \(K\) is a field containing \(\QQ\). Reading the fundamental theorem of
arithmetic as an isomorphism of monoids
\begin{equation}\label{eq:monoid}
  (\NN^+,\cdot)\;\cong \;\coprod_{p}(\NN,+)
\end{equation}
identifies \(\Gam\) with the large power series ring \(K[[x_1,x_2,\dots]]\) on
countably many variables, \(x_i\) corresponding to the \(i\)-th prime \(p_i\);
Cashwell and Everett proved that this ring is a unique factorisation domain.

In \cite{Snellman2000} the \emph{truncations}
\begin{equation}
  \Gam_n = \{\,f \in \Gam : f(m) = 0\text{ for all }m > n\,\}
\end{equation}
were studied. Writing \(r = \pi(n)\) and giving the monomial
\(x_1^{a_1}\cdots x_r^{a_r}\) the \emph{weight} \(\prod_i p_i^{a_i}\), so that
monomials correspond bijectively to positive integers, one has
\begin{equation}\label{eq:GammaSI}
  \Gam_n \cong S/I_n,\qquad S = K[x_1,\dots,x_r],
\end{equation}
with \(I_n\) generated by all
monomials of weight \(>n\). The ideal \(I_n\) is strongly stable with respect to
the \emph{reversed} order of the variables, because, whenever \(j > i\),
\begin{equation}\label{eq:stable}
  w\!\left(m\,x_j/x_i\right) = w(m)\,\frac{p_j}{p_i} > w(m);
\end{equation}
this single inequality
puts the Eliahou--Kervaire resolution \cite{EliahouKervaire} and the Golod
property \cite{AramovaHerzog,Peeva,Golod,GulliksenLevin} within reach, and
yields a rational Poincaré--Betti series for \(\Gam_n\) expressed in terms of
the numbers
\begin{equation}\label{eq:Cdef}
  C_{n,v} = \#\{\,m \in G(I_n) : \min(m) = v\,\},\qquad C_n = \#G(I_n),
\end{equation}
where \(G(I_n)\) is the minimal monomial generating set and \(\min(m)\) is the
least index occurring in \(m\).

The purpose of this note is threefold.

First, \emph{errata}. Theorem 3.5(3) of \cite{Snellman2000} is false, an
intermediate display in Example 4.5 there is miscomputed, and the proofs of
Theorem 3.4(4) and Theorem 3.4(5) are incomplete. Section \ref{sec:errata}
records this and supplies corrected statements and complete proofs. The false
statement is used nowhere else in \cite{Snellman2000}, and the example's final
answer is correct, so nothing else in that paper is affected.

Second, and principally, a \emph{sieve-theoretic identification}. Recall
Legendre's sifting function
\begin{equation}\label{eq:phidef}
  \Phi(x,y) = \#\{\,m \le x : p \mid m \implies p > y\,\},
\end{equation}
the number of integers up to \(x\) free of prime factors \(\le y\). Our main
result, Theorem~\ref{thm:main}, is that
\begin{equation}
  \boxed{\;C_{n,v} = \Phi(n,p_v).\;}
\end{equation}
Thus the generator counts governing a minimal free resolution over \(\Gam_n\)
are values of the classical sifting function, and the numerical
tables of \cite{Snellman2000} are tables of \(\Phi\). Theorem 3.2 of
\cite{Snellman2000} describes \(C_{n,v}\) as the cardinality of a
\emph{different} set, namely the \(p_v\)-rough numbers in the interval
\((n/p_v,n]\), and Theorem~\ref{thm:main} may be read as saying that this
set is equinumerous with the much more familiar one in \eqref{eq:phidef}.

Third, the \emph{consequences}. Section~\ref{sec:conjecture} proves Conjecture
4.6 of \cite{Snellman2000}, which asserted a precise shape for the reduced
Poincaré--Betti series and was left open there with numerical evidence for
\(n \le 25\). Section~\ref{sec:asymptotics} determines the asymptotics of
\(C_n\), showing that
\[
  C_n \sim \tfrac{1}{2}\pi(n)^{2} \sim \frac{n^{2}}{2\log^{2}n},
\]
and hence that the lower bound \(C_n \ge \binom{\pi(n)+1}{2}\) of Theorem 3.4(3) of
\cite{Snellman2000} is asymptotically an equality. Section~\ref{sec:next}
collects what remains open.

\begin{remark}\label{rem:m2}
  The Golod criterion of \cite{AramovaHerzog,Peeva} requires the ideal to be
  contained in \((x_1,\dots,x_r)^{2}\), and \(I_n\) is: a degree-one
  generator would be some \(x_i\) with \(p_i > n\), and no such variable occurs
  in \(S = K[x_1,\dots,x_r]\) since \(r = \pi(n)\). Equivalently, this is Theorem
  3.5(1) of \cite{Snellman2000}. We record it because it is easy to leave unchecked.
\end{remark}

\begin{remark}
  All numerical claims below were verified by computer, along a route
  independent of the formulas being tested: the minimal generators of \(I_n\)
  were enumerated directly and their least supports tabulated, rather than
  computed from any of the counting formulas. The code and its raw output are
  available in the repository \cite{gitrepo}.
\end{remark}

\section{Notation and the counting theorem}\label{sec:prelim}

Throughout, \(p_1 = 2 < p_2 = 3 < \cdots\) are the primes, \(\pi(n)\) is the number of
primes \(\le n\), and \(\lpf(m)\) denotes the least prime factor of \(m > 1\).
We abbreviate \(r = \pi(n)\) when \(n\) is fixed, this being the number of
variables of \(S\). (\cite{Snellman2000} writes \(r(n)\) throughout for what we
write \(\pi(n)\); we have changed the notation because the results below sit
alongside the sieve-theoretic literature, where \(\pi\) is universal. Quotations
from \cite{Snellman2000} below keep its own notation.)
For a monomial \(m = x_1^{a_1}\cdots x_r^{a_r}\) we write \(\min(m)\) and
\(\max(m)\) for the least and greatest \(i\) with \(a_i > 0\), and
\(|m|=\sum_i a_i\). For an integer \(m > 1\), \(\Om(m)\) denotes the number of
prime factors of \(m\) counted with multiplicity, so that
\[
  \Om(p_1^{a_1}\cdots p_r^{a_r}) = a_1+\dots+a_r,\qquad \Om(1) = 0;
\]
following
\cite{Snellman2000} we also write \(\lambda\) for \(\Om\). The companion
function counting prime factors \emph{without} multiplicity, usually written
\(\omega\), is not used in this note. The symbol \(\omega\) does appear in
Remark~\ref{rem:hypotheses} and in Section~\ref{sec:related}, and there it
always denotes Buchstab's function, an unrelated object.

We shall use the following description, which is Theorem 3.2 of
\cite{Snellman2000}.

\begin{theorem}[\cite{Snellman2000}, Thm.~3.2]\label{thm:S32}
  For \(1 \le v \le \pi(n)\),
  \begin{equation}\label{eq:S32}
    C_{n,v} = \#\{\,x \in \NN : n/p_v < x \le n \text{ and }
    p \mid x \implies p \ge p_v \,\}.
  \end{equation}
\end{theorem}

It will be convenient to name the counting function appearing implicitly there.
For a prime \(p\) and real \(y \ge 0\) put
\begin{equation}
  R_p(y) = \#\{\,m \le y : q \mid m \implies q \ge p\,\},
\end{equation}
the number of \(p\)-rough integers up to \(y\); the integer \(1\) is counted,
vacuously. Theorem~\ref{thm:S32} then reads
\begin{equation}\label{eq:CasR}
  C_{n,v} = R_{p_v}(n)-R_{p_v}(n/p_v).
\end{equation}
Note that
\begin{equation}\label{eq:RvsPhi}
  R_p(y) = \Phi(y,p)+\#\{\,m \le y : \lpf(m) = p\,\},
\end{equation}
so \(R_p\) and \(\Phi\) differ precisely by the multiples of \(p\) that are
\(p\)-rough.

\section{The main theorem}\label{sec:main}

\begin{theorem}\label{thm:main}
  For all \(n \ge 2\) and \(1 \le v \le \pi(n)\),
  \begin{equation}\label{eq:main}
    C_{n,v} = \Phi(n,p_v),
  \end{equation}
  the number of integers in \([1,n]\) having no prime factor \(\le p_v\).
\end{theorem}

\begin{proof}
  Write \(q = p_v\). Every \(q\)-rough integer \(m\) factors uniquely as
  \(m = q^{a}m'\) with \(a \ge 0\) and \(m'\) free of prime factors \(\le q\);
  conversely every such pair \((a,m')\) gives a \(q\)-rough integer. Counting
  those with \(m \le y\) according to \(a\) gives
  \begin{equation}\label{eq:Rsum}
    R_q(y) = \sum_{a \ge 0}\Phi\!\left(y/q^{a},\,q\right),
  \end{equation}
  a finite sum, since the terms vanish once \(q^{a} > y\). Substituting
  \eqref{eq:Rsum} into \eqref{eq:CasR}, the two sums telescope:
  \begin{equation}\label{eq:telescope}
  \begin{split}
    C_{n,v}&= R_q(n)-R_q(n/q)\\
           &= \sum_{a \ge 0}\Phi\!\left(\frac{n}{q^{a}},q\right)
            -\sum_{a \ge 0}\Phi\!\left(\frac{n}{q^{a+1}},q\right)\\
           &= \Phi(n,q).
  \end{split}
  \end{equation}
\end{proof}

\begin{example}\label{ex:n25}
  Take \(n = 25\), so \(r = 9\). The integers in \([1,25]\) with no prime factor
  \(\le 2\) are the thirteen odd numbers, so \(C_{25,1} = 13\); those with no
  prime factor \(\le 3\) are \(1,5,7,11,13,17,19,23,25\), so \(C_{25,2} = 9\);
  those with no prime factor \(\le 5\) are \(1,7,11,13,17,19,23\), so
  \(C_{25,3} = 7\). These agree with Figure 1 of \cite{Snellman2000}. By contrast,
  the set of Theorem~\ref{thm:S32} for \(v = 2\) is
  \(\{9,11,13,15,17,19,21,23,25\}\), which is a different set of the same size.
\end{example}

The same argument refines to the grading. Give \(S\) the standard grading,
\(\deg x_i = 1\), so that the monomial of weight \(m\) has degree \(\Om(m)\), and
let
\begin{equation}\label{eq:Cvddef}
  C_{n,v,d} = \#\{\,m \in G(I_n) : \min(m) = v,\ \deg m = d\,\}
\end{equation}
be the refined counts tabulated in Figure 2 of \cite{Snellman2000}. Let
\(\Phi(x,y;u)\) be the generating function
\begin{equation}\label{eq:phiu}
  \Phi(x,y;u) = \sum_{\substack{m \le x\\ p \mid m \implies p > y}} u^{\Om(m)}
    \in \ZZ[u].
\end{equation}
Note that the telescoping survives, and that specialising \(u = 1\) recovers
\(\Phi(x,y;1) = \Phi(x,y)\); here, the correction vanishes.

\begin{thmmainu}
  For all \(n \ge 2\) and \(1 \le v \le \pi(n)\), with \(q = p_v\),
  \begin{equation}\label{eq:graded}
    \sum_{d}C_{n,v,d}\,u^{d-1}
      = \Phi(n,q;u)-(1-u)\sum_{a \ge 1}u^{a-1}\,\Phi\!\left(n/q^{a},\,q;u\right).
  \end{equation}
\end{thmmainu}

\begin{proof}
  A monomial \(m\) with \(\min(m) = v\) lies in \(G(I_n)\) exactly when
  \(w(m) > n\) and \(w(m/x_i) \le n\) for every \(x_i \mid m\); the largest of the
  \(w(m/x_i)\) is \(w(m)/q\), since \(q\) is the least prime in the support. So
  these generators correspond bijectively to the \(q\)-rough \(x = w(m)/q\) in
  \((n/q,n]\), which form the set of Theorem~\ref{thm:S32}, and
  \(\deg m = 1+\Om(x)\).
  Hence
  \[
    \sum_{d}C_{n,v,d}\,u^{d-1} = R_q(n;u)-R_q(n/q;u),
  \]
  where \(R_q(y;u) = \sum u^{\Om(m)}\) over the \(q\)-rough \(m \le y\). Factoring
  \(m = q^{a}m'\) as in the proof of Theorem~\ref{thm:main}, now with
  \(u^{\Om(m)} = u^{a}u^{\Om(m')}\), gives
  \[
    R_q(y;u) = \sum_{a \ge 0}u^{a}\,\Phi(y/q^{a},q;u),
  \]
  and therefore
  \begin{equation}\label{eq:telescopeu}
  \begin{split}
    R_q(n;u)-R_q(n/q;u)
      &= \sum_{a \ge 0}u^{a}\,\Phi\!\left(\frac{n}{q^{a}},q;u\right)
        -\sum_{a \ge 1}u^{a-1}\,\Phi\!\left(\frac{n}{q^{a}},q;u\right)\\
      &= \Phi(n,q;u)+\sum_{a \ge 1}\bigl(u^{a}-u^{a-1}\bigr)
        \Phi\!\left(\frac{n}{q^{a}},q;u\right).
  \end{split}
  \end{equation}
\end{proof}

\begin{exampleprime}
  Take \(n = 30\) and \(v = 2\), so \(q = 3\). In the ungraded telescoping
  \eqref{eq:telescope} the terms cancel in pairs and leave
  \(C_{30,2} = \Phi(30,3) = 10\), the integers in \([1,30]\) prime to \(6\)
  being \(1,5,7,11,13,17,19,23,25,29\). In the graded telescoping
  \eqref{eq:telescopeu} the same terms carry the weights \(u^{a}\) and
  \(u^{a-1}\), and no longer cancel. Here
  \[
    \Phi(30,3;u) = 1+8u+u^{2},\qquad \Phi(10,3;u) = 1+2u,
  \]
  \(\Phi(30/9,3;u) = \Phi(30/27,3;u) = 1\), and the terms with \(a \ge 4\)
  vanish, so
  \[
    \sum_{d}C_{30,2,d}\,u^{d-1}
      = (1+8u+u^{2})-(1-u)\bigl((1+2u)+u+u^{2}\bigr)
      = 6u+3u^{2}+u^{3}.
  \]
  This is what the generators say: they are \(3x\) with \(x\) odd in
  \((10,30]\), namely the six primes \(11,13,17,19,23,29\) (degree \(2\)),
  \(15,21,25\) (degree \(3\)) and \(27\) (degree \(4\)). The leading term
  \(\Phi(30,3;u)\) alone would be wrong in every coefficient. Setting \(u = 1\)
  in the graded identity, the correction vanishes and both sides give
  \(1+8+1 = 6+3+1 = 10\), which is the ungraded count.
\end{exampleprime}

\begin{remark*}
  As the above example illustrates, at \(u = 1\) the correction vanishes and
  Theorem~\ref{thm:main} is what remains. Pinning the degree instead of carrying it, write \(\Phi_{k}(x,y)\) for
  the number of \(m \le x\) free of prime factors \(\le y\) with \(\Om(m) = k\).
  Comparing coefficients of \(u^{d-1}\),
  \[
    C_{n,v,d} = \Phi_{d-1}(n,q)-\sum_{a \ge 1}
      \Bigl(\Phi_{d-a}\!\left(n/q^{a},q\right)-\Phi_{d-a-1}\!\left(n/q^{a},q\right)\Bigr),
  \]
  so the natural guess \(C_{n,v,d} = \Phi_{d-1}(n,q)\) is false: the upper index
  shifts with \(a\), and the sums no longer cancel. It fails already at
  \(n = 4\), where the generator \(x_1^{3}\) of weight \(8\) gives
  \(C_{4,1,3} = 1\) while no odd integer \(\le 4\) has two prime factors.
  Theorem~\ref{thm:main}\({}'\) and the coefficient formula were checked against
  the minimal generators enumerated directly for every \(n \le 200\) and every
  \(v\), and against every entry of Figure 2 of \cite{Snellman2000}.
\end{remark*}

Splitting \([1,n]\) into \(1\), the primes and the composites gives at once the
following reformulation, which is the form we shall use.

\begin{corollary}\label{cor:split}
  For \(1 \le v \le \pi(n)\),
  \begin{equation}\label{eq:split}
    C_{n,v} = \bigl(\pi(n)-v+1\bigr)+E(n,v),
  \end{equation}
  where
  \begin{equation}\label{eq:Edef}
    E(n,v) = \#\{\,x \le n :
      x \text{ composite},\ p \mid x \implies p > p_v\,\}.
  \end{equation}
  In particular
  \begin{equation}\label{eq:lowerbound}
    C_{n,v}\ \ge \ \pi(n)-v+1,
  \end{equation}
  with equality if and only if \(p_{v+1}^{2} > n\).
\end{corollary}

\begin{proof}
  The integers counted by \(\Phi(n,p_v)\) are \(1\); the primes \(q\) with
  \(p_v < q \le n\), of which there are \(\pi(n)-v\); and the composites
  counted by \(E(n,v)\). For the last sentence, \(E(n,v) = 0\) exactly when there
  is no composite \(\le n\) all of whose prime factors exceed \(p_v\), and the
  least such composite is \(p_{v+1}^{2}\).
\end{proof}

Corollary~\ref{cor:split} contains Theorem 3.4(2) of \cite{Snellman2000} (put
\(v = 1+\pi(n)-u\)) and sharpens it to an exact criterion for equality. Summing over
\(v\) gives a closed form for the total.

\begin{corollary}\label{cor:total}
  For the total number of minimal generators, it holds that
  \begin{equation}\label{eq:Cntotal}
    C_n = \binom{\pi(n)+1}{2}+\sum_{\substack{x \le n\\ x\text{ composite}}}
        \bigl(\pi(\lpf x)-1\bigr).
  \end{equation}
\end{corollary}

\begin{proof}
  Sum Corollary~\ref{cor:split} over \(v = 1,\dots,r\). The first terms contribute
  \[
    \sum_{v = 1}^{r}(r-v+1) = \sum_{j = 1}^{r}j = \binom{r+1}{2}.
  \]
  For the second, exchange the order of summation: a composite \(x \le n\) is counted by \(E(n,v)\) for exactly those
  \(v \ge 1\) with \(p_v < \lpf x\), and there are \(\pi(\lpf x)-1\) of those.
\end{proof}

\begin{remark}
  Corollary~\ref{cor:total} proves the increment formula
  \[
    C_n-C_{n-1}=
    \begin{cases}
      \pi(n) & \text{if \(n\) is prime},\\[2pt]
      \pi(\lpf n)-1 & \text{if \(n\) is composite},
    \end{cases}
  \]
  recorded empirically by Costello for the OEIS sequence \seqnum{A182843}; see
  Section~\ref{sec:asymptotics}. The two cases have different sources: when
  \(n\) is prime the whole increment comes from \(\binom{\pi(n)+1}{2}\), which
  grows by \(\pi(n)\), and the sum over composites is unchanged; when \(n\) is
  composite the binomial is unchanged and the sum gains the single term
  \(\pi(\lpf n)-1\). They can be written as one expression,
  \([\,n\text{ prime}\,]+\bigl(\pi(\lpf n)-1\bigr)\) with an Iverson bracket,
  because \(\lpf n = n\) for \(n\) prime.
\end{remark}

\section{Errata for \texorpdfstring{\cite{Snellman2000}}{[Sne00]}}\label{sec:errata}

\subsection{A false statement}

Theorem 3.5(3) of \cite{Snellman2000} asserts
\begin{equation}\label{eq:false}
  \binom{\pi(n)}{2} = \#\{\,m \in \NN^+ : m \le n,\ \lambda(m) = 2\,\},
\end{equation}
and the proof there reads, in full, ``The first and the last assertions are
obvious.'' Statement~\eqref{eq:false} is false. Its left-hand side counts
\emph{pairs of distinct primes \(\le n\)}, its right-hand side counts
\emph{products of two primes that are \(\le n\)}; the two agree only for
\(n = 2\) and \(n = 4\). At \(n = 10\) the left side is \(\binom{4}{2} = 6\) while the
right side counts \(\{4,6,9,10\}\), so is \(4\). Asymptotically the left side is
\(\sim n^{2}/(2\log^{2}n)\) and the right side \(\sim n\log\log n/\log n\).

The intended statement was presumably the following, which is what the proof of
Theorem 3.4(2)--(3) of \cite{Snellman2000} actually establishes.

\begin{proposition}\label{prop:fixed35}
  For every \(n \ge 2\),
  \begin{equation}\label{eq:supp2}
    \#\{\,m \in G(I_n) : |\!\operatorname{supp}(m)|=2\,\}
      \ \ge \ \binom{\pi(n)}{2},
  \end{equation}
  with equality exactly for \(2 \le n \le 8\).
\end{proposition}

\begin{proof}
  Let \(1 \le i < j \le r\). Choose \(b \ge 1\) with \(p_i^{\,b-1} \le n < p_i^{\,b}\).
  Then
  \[
    w(x_i^{\,b-1}x_j) = p_i^{\,b-1}p_j > p_i^{\,b-1}p_i = p_i^{\,b} > n,
  \]
  so \(x_i^{\,b-1}x_j \in I_n\) and is therefore divisible by some
  \(m \in G(I_n)\), necessarily with \(\operatorname{supp}(m) \subseteq \{i,j\}\).
  Now \(x_i^{\,b-1} \notin I_n\) because \(p_i^{\,b-1} \le n\), so \(m\) is not a
  power of \(x_i\); and \(x_j \notin I_n\) because \(p_j \le n\), so \(m\) is not
  a power of \(x_j\). Hence \(\operatorname{supp}(m) = \{i,j\}\). Distinct pairs
  give distinct generators, whence the inequality.

  For the equality range, count the generators supported on \(\{1,2\}\)
  exactly. Let
  \[
    b^{*} = \max\{\,b \ge 1 : 3^{b} \le n\,\},\qquad
    a(b) = \min\{\,a \ge 1 : 2^{a}3^{b} > n\,\}
  \]
  for \(1 \le b \le b^{*}\). Then
  \(x_1^{a(b)}x_2^{b}\) is a minimal generator: its weight exceeds \(n\) by
  definition; \(2^{a(b)-1}3^{b} \le n\) by minimality of \(a(b)\) when
  \(a(b) \ge 2\), and because \(3^{b} \le n\) when \(a(b) = 1\); and
  \[
    2^{a(b)}3^{\,b-1} = \tfrac{1}{3} \cdot 2^{a(b)}3^{b}
      \le \tfrac{1}{3} \cdot 2 \cdot 2^{a(b)-1}3^{b} \le \tfrac{2}{3}n \le n .
  \]
  Conversely every generator supported on \(\{1,2\}\) arises this way, since
  \(2^{a-1}3^{b} \le n\) forces \(b \le b^{*}\), and \(a\) is then determined.
  So there are exactly \(b^{*}\) of them.

  Hence the inequality is strict as soon as \(b^{*} \ge 2\), that is, as soon as
  \(n \ge 9\); and for \(2 \le n \le 8\) equality is a finite check.
\end{proof}

\subsection{A miscomputed example}

Example 4.5 of \cite{Snellman2000} works out the case \(n = 5\), where \(r = 3\),
\(S = K[x_1,x_2,x_3]\) and
\(I_5 = (x_1^{3},x_1x_2,x_1x_3,x_2^{2},x_2x_3,x_3^{2})\). In the notation of that
paper, \(P^{S}_{I}(t,u)\) for the bigraded Poincaré--Betti series of \(I\) over
\(S\), its \emph{graded} display reads
\begin{equation}\label{eq:badex}
  P^{S}_{I_5}(t,u) = u^{2}+2u^{2}(1+t)+(2u^{2}+u^{3})(1+t)^{2},
\end{equation}
with the factor \((1+t)\). Corollary 4.2 of the same paper has \((1+tu)\), and
that is the correct factor: a minimal generator of degree \(j\) contributes to
homological degree \(i\) in \emph{internal} degree \(i+j\), so the variable
marking internal degree must accompany \(t\). Restoring it,
\begin{equation}\label{eq:goodex}
\begin{split}
  P^{S}_{I_5}(t,u) &= u^{2}+2u^{2}(1+tu)+(2u^{2}+u^{3})(1+tu)^{2}\\
    &= 5u^{2}+u^{3}+(6u^{3}+2u^{4})t+(2u^{4}+u^{5})t^{2},
\end{split}
\end{equation}
whereas \eqref{eq:badex} expands to
\(5u^{2}+u^{3}+(6u^{2}+2u^{3})t+(2u^{2}+u^{3})t^{2}\). The corrected
coefficients are the graded Betti numbers of \(I_5\), namely
\(\beta_{0,2} = 5\), \(\beta_{0,3} = 1\), \(\beta_{1,3} = 6\), \(\beta_{1,4} = 2\),
\(\beta_{2,4} = 2\) and \(\beta_{2,5} = 1\).

Two remarks on the scope of the error. First, \emph{the example's conclusion is
correct}: the value it prints for \(P^{A_5}_{K}(t,u)\) is the one obtained from
the corrected series through the graded Golod formula of
\cite[Cor.~4.4]{Snellman2000}, not from \eqref{eq:badex}, and it agrees with
the \(n = 5\) row of Figure 3 of \cite{Snellman2000}, which is likewise correct.
Only the intermediate display is wrong. Second, setting \(u = 1\) in either
version gives \(6+8t+3t^{2}\), the ungraded series that the same example
prints, which is why the discrepancy survives any check made after that
specialisation.

\subsection{Two incomplete proofs}

\emph{Theorem 3.4(4) of \cite{Snellman2000}} asserts that for each \(v\) one has
\(C_{n,1+\pi(n)-v} = v\) for all sufficiently large \(n\). The proof given there
reads ``the only integers \(x \le n\) with all prime factors \(\ge 1+r(n)-v\) are
\(p_{1+r(n)-v},\dots,p_{r(n)}\)\dots{} and they are all \(>n/p_v\)''. Two
subscripts are wrong: the condition on prime factors should read
\(\ge p_{1+r(n)-v}\) (a prime, not an index), and the final bound should read
\(>n/p_{1+r(n)-v}\). With Corollary~\ref{cor:split} the statement becomes
immediate, and effective.

\begin{proposition}\label{prop:fixed344}
  Fix \(v \ge 1\). Then \(C_{n,1+\pi(n)-v} = v\) for every \(n > 4^{\,v-1}\).
\end{proposition}

\begin{proof}
  Put \(u = 1+\pi(n)-v\). By Corollary~\ref{cor:split}, \(C_{n,u} = v\) if and only
  if \(p_{u+1}^{2} > n\), where \(p_{u+1} = p_{\pi(n)-v+2}\). By Bertrand's postulate
  there is a prime in \((y,2y)\) for every \(y > 1\); applied repeatedly this
  gives \(p_{\pi(n)-j+1} > n/2^{\,j}\) for \(j \ge 1\), while for \(j = 0\) the same
  inequality reads \(p_{\pi(n)+1} > n\), true by definition of \(\pi\). So
  \(p_{\pi(n)-v+2} > n/2^{\,v-1}\) and hence \(p_{u+1}^{2} > n^{2}/4^{\,v-1}\), which
  exceeds \(n\) as soon as \(n > 4^{\,v-1}\).
\end{proof}

\emph{Theorem 3.4(5) of \cite{Snellman2000}} asserts that \(C_{n,v} = C_{n-1,v}\)
for all \(v\) when \(n\) is even. The proof there establishes only one of the two
inclusions needed. It shows that an \(x\) counted for \(n\) is counted for
\(n-1\); the converse, that no \(x\) is counted for \(n-1\) but not for \(n\),
is not addressed. The statement is nonetheless true, and with Theorem
\ref{thm:main} the whole matter is a triviality.

\begin{proposition}\label{prop:fixed345}
  If \(n \ge 4\) is even then \(C_{n,v} = C_{n-1,v}\) for \(1 \le v \le \pi(n)\); for
  \(v > \pi(n)\) both sides vanish.
\end{proposition}

\begin{proof}
  By Theorem~\ref{thm:main}, \(C_{n,v} = \Phi(n,p_v)\) and
  \(C_{n-1,v} = \Phi(n-1,p_v)\). These differ by \(1\) or \(0\) according as
  \(n\) is or is not counted by \(\Phi(\cdot,p_v)\); and \(n\) is even, so it
  is divisible by \(p_1 = 2 \le p_v\) and is never counted.
\end{proof}

\begin{remark}
  For completeness we note that the argument omitted in \cite{Snellman2000} can
  also be supplied directly, \emph{for \(v > 1\)}: an integer lying in
  \(((n-1)/p_v,\,n/p_v]\) is unique if it exists, since that interval has length
  \(1/p_v < 1\), and it must satisfy \(p_vx = n\); as \(n\) is even and \(p_v\) odd
  for \(v > 1\), such an \(x\) is even and so is excluded by the condition on
  prime factors. At \(v = 1\) this fails: the roughness condition is then vacuous,
  and the two sets of Theorem~\ref{thm:S32} differ, neither containing
  the other, their cardinalities agreeing only after a compensating gain at the
  top end. That case is \cite{Snellman2000}'s own item (6),
  \(C_{n,1} = \lceil n/2\rceil\), and needs no repair.
\end{remark}

\subsection{Smaller items}

For the record: in equation (15) of \cite{Snellman2000} the range should be
\(1 \le j \le r\), not \(1 \le j \le n\), since \(j\) indexes the variables
\(x_1,\dots,x_r\); in Theorem 3.2 and equation (19) the solution vector is
written \((b_1,\dots,b_r) \in \NN^{r}\) although only \(b_v,\dots,b_r\) are
constrained, so that as literally stated the solution set is infinite for
\(v > 1\), although the proof supplies the missing condition \(b_j = 0\) for
\(j < v\) immediately and the corresponding statement in Theorem 1.2(III) is correct as
printed; the caption of Figure 2 names the quantities \(C_{n,i,g}\) where the
text defines \(C_{n,v,d}\); and Theorem 4.3 is attributed to
``Herzog--Aramova'' in the text but ``Aramova--Herzog'' in the abstract and
bibliography. Finally, we take the opportunity to note a small looseness of
attribution. Theorem 4.3 of \cite{Snellman2000} is stated for \emph{stable}
ideals and credited jointly to \cite{AramovaHerzog} and \cite{Peeva}. The
result of \cite{Peeva} (Corollary 1.2 there) assumes the ideal to be
\(0\)-Borel fixed, that is \emph{strongly} stable, and it is in that form that
the joint attribution is standard; see \cite[Thm.~6.16(6)]{McCulloughPeeva}.
The version for merely stable ideals is obtained in \cite{AramovaHerzog},
\S2, where the Koszul cycles exhibited for a stable ideal are observed to have
pairwise vanishing products. None of this affects \cite{Snellman2000}, whose
Proposition 2.6 proves \(I_n\) strongly stable, so that both hypotheses hold;
we note it only because the distinction matters if the argument is reused. An
alternative route to the same conclusion runs through componentwise linearity
\cite{HerzogReinerWelker}, since stable ideals have linear quotients.

\section{Conjecture 4.6}\label{sec:conjecture}

By Corollary 4.4 of \cite{Snellman2000}, the Poincaré--Betti series of the
residue field over \(A_n = \Gam_n\) is
\begin{equation}\label{eq:PB}
  P^{A_n}_K(t) = \frac{(1+t)^{r}}{D_n(t)},\qquad
  D_n(t) = 1-t^{2}\sum_{j = 1}^{r}(1+t)^{\,r-j}C_{n,j},
\end{equation}
with \(r = \pi(n)\). Conjecture 4.6 of \cite{Snellman2000} asserted that this
fraction reduces to
\begin{equation}\label{eq:conj}
  P^{A_n}_K(t) = -\frac{(1+t)^{\ell_1(n)}}{q_n(t)},\qquad
  q_n(t) = \sum_{i = 0}^{\ell_2(n)}h_i(n)t^{i},
\end{equation}
where
\begin{enumerate}
  \item \(q_n(-1) \ne 0\);
  \item \(\ell_1(n) = \#\{\,p \text{ odd prime} : p^{2} \le n\,\}\);
  \item \(\ell_2(n) = \ell_1(n)+1\);
  \item \(h_0(n) = -1\);
  \item \(h_1(n) = \pi(n)-\ell_1(n)\);
  \item \(h_{\ell_2(n)}(n) = C_{n,1} = \lceil n/2\rceil\).
\end{enumerate}
We prove all six.

The point of the next lemma is that the whole conjecture is controlled by a
single quantity: the order of vanishing of \(D_n\) at \(t = -1\).

\begin{lemma}\label{lem:secondiff}
  Put \(u = 1+t\) and \(c_k = C_{n,r-k}\) for \(0 \le k \le r-1\), with \(c_k = 0\)
  otherwise. Then the coefficient of \(u^{0}\) in \(D_n\) is \(1-c_0\), and for
  \(k \ge 1\) the coefficient of \(u^{k}\) is
  \(-\left(c_k-2c_{k-1}+c_{k-2}\right)\).
\end{lemma}

\begin{proof}
  Reindexing \eqref{eq:PB} by \(k = r-j\) gives
  \(D_n = 1-(u-1)^{2}T(u)\) with \(T(u) = \sum_{k = 0}^{r-1}c_ku^{k}\); expand
  \((u-1)^{2} = u^{2}-2u+1\).
\end{proof}

\begin{lemma}\label{lem:c0}
  \(c_0 = C_{n,\pi(n)} = 1\) for every \(n \ge 2\). Consequently \(D_n(-1) = 0\).
\end{lemma}

\begin{proof}
  \(\Phi(n,p_r)\) counts the integers in \([1,n]\) with no prime factor
  \(\le p_r\); as \(p_r\) is the largest prime \(\le n\), the only such integer
  is \(1\). Now apply Theorem~\ref{thm:main} and Lemma~\ref{lem:secondiff}.
\end{proof}

\begin{proposition}\label{prop:order}
  The order of vanishing of \(D_n(t)\) at \(t = -1\) equals \(\pi(n)-\ell_1(n)\).
\end{proposition}

\begin{proof}
  By Lemmas~\ref{lem:secondiff} and \ref{lem:c0} the order is the least
  \(k \ge 1\) with \(c_k-2c_{k-1}+c_{k-2} \ne 0\). Since \(c_{-1} = 0\) and
  \(c_0 = 1\), an induction shows that the second differences vanish for all
  \(k \le K\) precisely when \(c_k = k+1\) for all \(k \le K\). So the order is the
  least \(k \ge 1\) with
  \[
    c_k \ne k+1,\qquad\text{that is,}\qquad C_{n,r-k} \ne r-(r-k)+1 .
  \]

  By Corollary~\ref{cor:split},
  \[
    C_{n,v} \ne r-v+1
    \iff p_{v+1}^{2} \le n
    \iff p_{v+1} \le \sqrt n .
  \]
  Suppose first that some admissible \(v\) has this property; since
  \(p_2^{2} = 9\), that requires \(n \ge 9\). The largest such \(v\) is then
  \[
    v = \pi\bigl(\lfloor\sqrt n\rfloor\bigr)-1
     =\#\{\,p\text{ odd} : p^{2} \le n\,\} = \ell_1(n)
     \qquad(n \ge 9),
  \]
  the two counts differing only by the prime \(2\). Hence the least admissible
  \(k\) is \(r-\ell_1(n)\). (The identity between the two counts in fact holds
  for every \(n \ge 4\); what needs \(n \ge 9\) is that the set of such \(v\) is
  non-empty.)

  If no such \(v\) exists (equivalently \(\ell_1(n) = 0\), equivalently
  \(n \le 8\)), then \(c_k = k+1\) for all
  \(k \le r-1\), and the first non-vanishing coefficient is that of \(u^{r}\),
  namely
  \[
    -\bigl(0-2r+(r-1)\bigr) = r+1 \ne 0;
  \]
  the formula \(r-\ell_1(n) = r\) still holds.
\end{proof}

\begin{theorem}\label{thm:conj}
  Let \(n \ge 2\). Conjecture 4.6 of \cite{Snellman2000} holds: with
  \(\ell_1(n)\) and \(\ell_2(n)\) as above, \eqref{eq:conj} holds together with
  clauses (1)--(6). (At \(n = 1\) the statement degenerates: \(r = 0\), \(D_1 = 1\),
  and clause (3) fails, \(\deg q_1 = 0 \ne 1 = \ell_2(1)\).)
\end{theorem}

\begin{proof}
  By Proposition~\ref{prop:order} we may write
  \(D_n(t) = (1+t)^{\,r-\ell_1}d(t)\) with \(d(-1) \ne 0\); set \(q_n = -d\). Then
  \eqref{eq:PB} gives
  \[
    P^{A_n}_K
      =\frac{(1+t)^{r}}{(1+t)^{\,r-\ell_1}\,(-q_n)}
      =-\frac{(1+t)^{\ell_1}}{q_n},
  \]
  which is \eqref{eq:conj} together with clause (2). Clause (1) is
  \(d(-1) \ne 0\).

  For the remaining clauses, note first that \(D_n(0) = 1\) and \(D_n'(0) = 0\),
  since \(D_n-1 = -t^{2}(\cdots)\) vanishes to order \(2\) at \(t = 0\); and that
  \(\deg D_n = r+1\) with leading coefficient \(-C_{n,1}\), the top-degree term
  of \eqref{eq:PB} arising only from \(j = 1\).

  (3) We have
  \[
    \deg q_n = \deg D_n-(r-\ell_1) = (r+1)-(r-\ell_1) = \ell_1+1 = \ell_2 .
  \]

  (4) \(h_0 = q_n(0) = -D_n(0)/1^{\,r-\ell_1} = -1\).

  (5) Differentiating \(q_n = -D_n(1+t)^{-(r-\ell_1)}\) at \(t = 0\) gives
  \[
    h_1 = q_n'(0) = -D_n'(0)+(r-\ell_1)D_n(0) = r-\ell_1 .
  \]

  (6) The leading coefficient of \(q_n\) is minus that of \(D_n\), namely
  \(C_{n,1}\), which is \(\Phi(n,2) = \lceil n/2\rceil\) by Theorem
  \ref{thm:main}.
\end{proof}

\begin{remark}
  Clause (2), the identification of the numerator exponent, was the whole
  difficulty, and Proposition~\ref{prop:order} shows why: it is the assertion
  that the sequence \(v\mapsto C_{n,v}\) is exactly linear, of slope \(-1\), on
  the range \(v > \ell_1(n)\), and departs from linearity at \(v = \ell_1(n)\).
  Corollary~\ref{cor:split} makes both halves of that assertion visible at
  once, because it measures the departure from linearity by \(E(n,v)\), a count
  of composites all of whose prime factors exceed \(p_v\), which is positive
  exactly when
  \(p_{v+1}^{2} \le n\). Figure~\ref{fig:profile} shows the effect.
\end{remark}

\begin{figure}[ht]
  \centering
  \includegraphics[width=\linewidth]{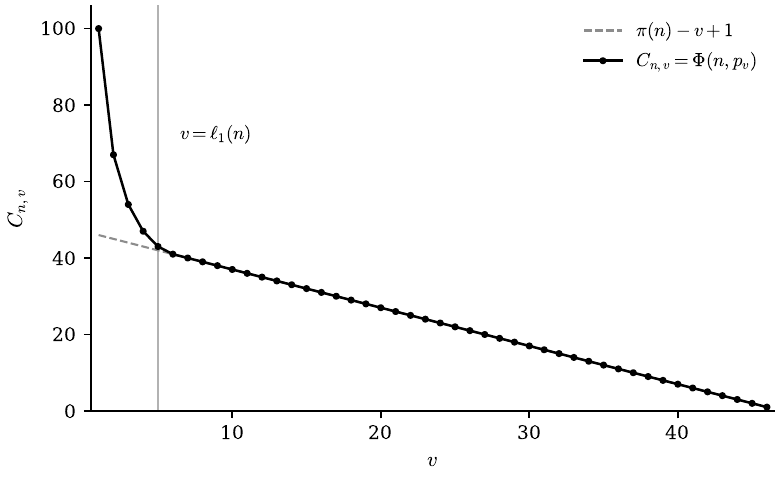}
  \caption{The profile \(v\mapsto C_{n,v} = \Phi(n,p_v)\) for \(n = 200\), against
    the line \(\pi(n)-v+1\) of Corollary~\ref{cor:split}. The vertical distance
    between them is \(E(n,v)\), the number of composites \(\le n\) all of whose
    prime factors exceed \(p_v\); here \(\pi(n) = 46\) and \(\ell_1(n) = 5\), and the two coincide
    for every \(v > 5\). That the profile is exactly linear of slope \(-1\)
    beyond \(v = \ell_1(n)\), and only there, is what Proposition
    \ref{prop:order} extracts as the order of vanishing of \(D_n\) at
    \(t = -1\).}\label{fig:profile}
\end{figure}

\section{The asymptotics of \texorpdfstring{\(C_n\)}{C\_n}}\label{sec:asymptotics}

Theorem 3.4(3) of \cite{Snellman2000} gives the lower bound
\(C_n \ge \binom{\pi(n)+1}{2}\), proved there by exhibiting one minimal generator
for each support of size \(1\) or \(2\). We show that essentially nothing else
contributes in the limit.

\begin{theorem}\label{thm:asymp}
  For \(n \ge 4\),
  \begin{equation}\label{eq:Cnbounds}
    0 \le C_n-\binom{\pi(n)+1}{2} \le n\bigl(\pi(\sqrt n)-1\bigr)
      =O\!\left(n^{3/2}/\log n\right),
  \end{equation}
  and consequently
  \begin{equation}\label{eq:Cnasymp}
    C_n \sim \binom{\pi(n)+1}{2} \sim \frac{\pi(n)^{2}}{2} \sim \frac{n^{2}}{2\log^{2}n}.
  \end{equation}
  In particular the bound of Theorem 3.4(3) of \cite{Snellman2000} is
  asymptotically an equality.
\end{theorem}

\begin{proof}
  By Corollary~\ref{cor:total} the difference equals
  \[
    C_n-\binom{\pi(n)+1}{2} = \sum_{x \le n\text{ composite}}\bigl(\pi(\lpf x)-1\bigr),
  \]
  a sum of non-negative terms. If \(x \le n\) is composite then
  \(\lpf x \le \sqrt x \le \sqrt n\), so each term is at most \(\pi(\sqrt n)-1\), and
  there are fewer than \(n\) composites \(\le n\); this gives the upper bound.
  The prime number theorem gives \(\pi(\sqrt n) \sim 2\sqrt n/\log n\), whence
  \[
    n\bigl(\pi(\sqrt n)-1\bigr) = O\!\left(\frac{n^{3/2}}{\log n}\right),
  \]
  while
  \[
    \binom{r+1}{2} \sim \frac{\pi(n)^{2}}{2} \sim \frac{n^{2}}{2\log^{2}n},
  \]
  which dominates it.
\end{proof}

\begin{corollary}
  \(C_n=\)\;\seqnum{A182843}\((n+1)\).
\end{corollary}

\begin{proof}
  \seqnum{A182843}\((m)\) is defined as the number of composite integers
  \(\ge m\) all of whose proper divisors are less than \(m\). A composite \(W\)
  has largest proper divisor \(W/\lpf W\), so the defining condition for
  \(m = n+1\) reads \(W > n\) and \(W \le n\lpf W\). These are exactly the \(W\)
  corresponding to minimal generators of \(I_n\): a monomial of weight \(W\)
  lies in \(G(I_n)\) precisely when \(W > n\) and \(W/p \le n\) for every prime
  \(p \mid W\). Of those last conditions only one needs checking, namely the one
  at \(p = \lpf W\), because \(W/p\) is largest for the smallest \(p\); so
  \(W/p \le n\) for all \(p \mid W\) is equivalent to \(W/\lpf W \le n\), that is,
  to \(W \le n\lpf W\).
\end{proof}

The upper bound in Theorem~\ref{thm:asymp} is far from best possible: almost
all integers are even and contribute nothing. We now determine the error term
exactly. Write
\begin{equation}\label{eq:Sdef}
  S(n) = C_n-\binom{\pi(n)+1}{2}
      =\sum_{x \le n\text{ composite}}\bigl(\pi(\lpf x)-1\bigr),
\end{equation}
and let \(S_2(n)\) be the part of the sum coming from \(x\) with \(\Om(x) = 2\).

\begin{lemma}\label{lem:semiprime}
  \(\displaystyle
    S_2(n) = \sum_{p \le \sqrt n}\bigl(\pi(p)-1\bigr)
             \bigl(\pi(n/p)-\pi(p)+1\bigr)\)
  (the sum over primes), and
  \(0 \le S(n)-S_2(n)\ll n^{4/3}/\log^{2}n\).
\end{lemma}

\begin{proof}
  If \(\Om(x) = 2\) then \(x = pq\) with \(p \le q\) prime, \(\lpf x = p\) and
  \(p \le \sqrt n\); for fixed \(p\) the admissible \(q\) are the primes in
  \([p,n/p]\), of which there are \(\pi(n/p)-\pi(p)+1\), and each contributes
  the weight \(\pi(p)-1\). If instead \(\Om(x) \ge 3\) then
  \((\lpf x)^{3} \le x \le n\), so \(\lpf x \le n^{1/3}\), and \(x\) is a multiple
  of \(\lpf x\); hence
  \[
    0 \le S(n)-S_2(n) \le \sum_{p \le n^{1/3}}\bigl(\pi(p)-1\bigr)\frac{n}{p}
    \ll n\sum_{p \le n^{1/3}}\frac 1{\log p}\ll\frac{n^{4/3}}{\log^{2}n},
  \]
  by \(\pi(p)\ll p/\log p\) and then
  \(\sum_{p \le Y}1/\log p\ll Y/\log^{2}Y\).
\end{proof}

Throughout what follows we abbreviate
\begin{equation}\label{eq:XM}
  X = \sqrt n,\qquad M = \log X = \tfrac{1}{2}\log n,\qquad P = \pi(\sqrt n),
\end{equation}
and we write
\begin{equation}\label{eq:gdef}
  g(t) = \bigl(\pi(t)-1\bigr)\bigl(\pi(n/t)-\pi(t)+1\bigr).
\end{equation}
With this notation Lemma~\ref{lem:semiprime} reads
\begin{equation}\label{eq:S2asg}
  S_2(n) = \sum_{p \le X}g(p),
\end{equation}
the sum being over primes.

The next lemma is the reason the rest is short. It removes a full third of the
asymptotics from the analysis, replacing it by an exact elementary sum.

\begin{lemma}[Exact splitting]\label{lem:split}
  Let \(p_1 < p_2 < \dots < p_P\) be the primes \(\le \sqrt n\), and put
  \begin{equation}\label{eq:Tdef}
    T(n) = \sum_{j = 1}^{P}(j-1)\,\pi(n/p_j).
  \end{equation}
  Then
  \begin{equation}\label{eq:split-exact}
    S_2(n) = T(n)-\frac{(P-1)P(2P-1)}{6}.
  \end{equation}
\end{lemma}

\begin{proof}
  Since \(p_j\) is the \(j\)-th prime, \(\pi(p_j) = j\). So the summand of
  \eqref{eq:S2asg} at \(p = p_j\) is
  \[
    g(p_j) = (j-1)\bigl(\pi(n/p_j)-j+1\bigr) = (j-1)\pi(n/p_j)-(j-1)^{2},
  \]
  and summing over \(j\),
  \[
    S_2(n) = T(n)-\sum_{j = 1}^{P}(j-1)^{2}
          =T(n)-\sum_{i = 1}^{P-1}i^{2}
          =T(n)-\frac{(P-1)P(2P-1)}{6}. \qedhere
  \]
\end{proof}

\begin{remark}\label{rem:whytwothirds}
  Lemma~\ref{lem:split} explains where the constant \(\tfrac{2}{3}\) comes from.
  Dividing \eqref{eq:split-exact} by \(P^{3}\) and expanding the second term
  \emph{exactly},
  \begin{equation}\label{eq:sqexact}
    \frac{S_2(n)}{P^{3}} = \frac{T(n)}{P^{3}}-\frac{1}{3}+\frac 1{2P}-\frac 1{6P^{2}} .
  \end{equation}
  So \(\tfrac{2}{3} = 1-\tfrac{1}{3}\). The \(1\) is the one thing that still
  requires the prime number theorem, namely \(T(n) \sim P^{3}\)
  (Proposition~\ref{prop:T} below), and the \(\tfrac{1}{3}\) is
  \(\sum i^{2} \sim P^{3}/3\), which requires nothing at all.
\end{remark}

The substitution governing \(T\) is \(p = Xe^{-a}\), under which \(n/p = Xe^{a}\);
the parameter \(a = \log(X/p) \ge 0\) measures how far below \(\sqrt n\) the
smaller prime factor lies. Note that the summand of \(T\) is, to leading order,
\emph{constant} in \(a\): the two factors \(\pi(p)-1 \approx (X/M)e^{-a}\) and
\(\pi(n/p) \approx (X/M)e^{a}\) have reciprocal exponentials. This is why no
Riemann sum is needed below, only a count of the primes involved.

\begin{lemma}\label{lem:pnt-unif}
  Let \(A_0\) be an absolute constant for which
  \(|\pi(y)-y/\log y|\le A_0\,y/\log^{2}y\) \((y \ge 2)\). There is an absolute
  constant \(A_1\) such that for every \(W \ge 1\), every \(n\) with
  \(M \ge 2W\), and every real \(a \in [0,W]\),
  \begin{equation}\label{eq:pnt-sub}
    \begin{aligned}
      \pi\bigl(Xe^{-a}\bigr)
        &= \frac{X}{M}\left(e^{-a}+\theta_1\,\frac{A_1(W+1)}{M}\right),\\
      \pi\bigl(Xe^{a}\bigr)
        &= \frac{X}{M}\left(e^{a}+\theta_2\,\frac{A_1(W+1)e^{W}}{M}\right),
    \end{aligned}
  \end{equation}
  where \(|\theta_1|,|\theta_2|\le 1\).

  The error constant in the second line grows with \(W\), because \(e^{a}\)
  itself does; this is what limits how fast \(W\) may be allowed to grow with
  \(n\).
\end{lemma}

\begin{proof}
  For \(y = Xe^{\mp a}\) we have \(\log y = M\mp a\), and \(0 \le a \le W \le M/2\)
  gives \(\tfrac{1}{2}M \le \log y \le \tfrac{3}{2}M\); in particular \(y \ge \sqrt X \ge M \ge 2\),
  so the hypothesis on \(A_0\) applies. We also use \(X \ge M^{2}\) freely below,
  which is automatic rather than an assumption: \(X = e^{M}\), and \(e^{M}/M^{2}\)
  attains its minimum \(e^{2}/4 > 1\) at \(M = 2\).

  \emph{First line.} By hypothesis on \(A_0\),
  \[
    \pi\bigl(Xe^{-a}\bigr)
      =\frac{Xe^{-a}}{M-a}+\theta\,A_0\frac{Xe^{-a}}{(M-a)^{2}},
      \qquad|\theta|\le 1 ,
  \]
  and since \(M-a \ge \tfrac{1}{2}M\) and \(e^{-a} \le 1\) the second term is at most
  \(4A_0X/M^{2}\). For the first, \(0 \le a/M \le \tfrac{1}{2}\) gives
  \[
    \frac{e^{-a}}{1-a/M}-e^{-a} = e^{-a} \cdot \frac{a/M}{1-a/M},
    \qquad\left|\;\cdot \;\right|\le \frac{2a}{M} \le \frac{2W}{M},
  \]
  so the first line holds with \(A_1 \ge 2+4A_0\).

  \emph{Second line.} The same computation with \(\log y = M+a\), now using the
  trivial \(M+a \ge M\) together with \(e^{a} \le e^{W}\), bounds the two error
  terms by \(A_0e^{W}/M\) and \(We^{W}/M\), hence by \((W+A_0)e^{W}/M\) in
  total, after division by \(X/M\).
\end{proof}

\begin{lemma}\label{lem:tail}
  There is an absolute constant \(c_0\) such that, for every \(W\) with
  \(1 \le W \le M/2\),
  \begin{equation}\label{eq:tailbound}
    \sum_{p \le Xe^{-W}}\pi(p)\,\pi(n/p)
      \ \le \ c_0\,e^{-W}\,\frac{X^{3}}{M^{3}}
      \ +\ c_0\,\frac{X^{3}}{M^{3}} \cdot \frac{\log n}{n^{1/4}} .
  \end{equation}
  The same bound holds for \(\sum_{p \le Xe^{-W}}g(p)\) and for
  \(\sum_{p \le Xe^{-W}}(\pi(p)-1)\pi(n/p)\), both summands being at most
  \(\pi(p)\pi(n/p)\).
\end{lemma}

\begin{proof}
  Let \(c_1\) be an absolute constant with \(\pi(y) \le c_1y/\log y\)
  (Chebyshev), and write \(L = \log n = 2M\). Split at \(n^{1/4}\).

  \emph{Range \(p \le n^{1/4}\).} Here \(\log(n/p) \ge \tfrac{3}{4}L\), so
  \[
    \pi(p)\,\pi(n/p)\ \le \ c_1^{2}\,\frac{n}{\log p \cdot \log(n/p)}
      \ \le \ \frac{4c_1^{2}}{3\log 2} \cdot \frac{n}{L} ,
  \]
  and there are at most \(\pi(n^{1/4}) \le 4c_1n^{1/4}/L\) such primes, giving a
  constant times \(n^{5/4}/L^{2}\); relative to \(X^{3}/M^{3} = 8n^{3/2}/L^{3}\)
  that is a constant times \(L/n^{1/4}\).

  \emph{Range \(n^{1/4} < p \le Xe^{-W}\).} Here \(\log p \ge \tfrac{1}{4}L\) and
  \(\log(n/p) \ge \tfrac{1}{2}L\), so \(\pi(p)\pi(n/p) \le 8c_1^{2}n/L^{2}\), while
  \[
    \pi\bigl(Xe^{-W}\bigr)\ \le \ c_1\frac{Xe^{-W}}{M-W}
      \ \le \ 2c_1\frac{Xe^{-W}}{M} = 4c_1\frac{Xe^{-W}}{L}
  \]
  by \(W \le M/2\). The product is at most
  \(32c_1^{3}e^{-W}n^{3/2}/L^{3} = 4c_1^{3}e^{-W}X^{3}/M^{3}\).
\end{proof}

\begin{proposition}\label{prop:T}
  There are absolute constants \(C_0\) and \(n_0\) such that for \(n \ge n_0\)
  \begin{equation}\label{eq:Tasymp}
    \left|\;\frac{M^{3}}{X^{3}}\,T(n)-1\;\right|
      \ \le \ C_0\,\frac{\log M}{\sqrt M} .
  \end{equation}
\end{proposition}

\begin{proof}
  Fix \(W \ge 1\) with \(M \ge 2W\), to be chosen. Split \(T\) at \(Xe^{-W}\); by
  Lemma~\ref{lem:tail} the part with \(p \le Xe^{-W}\) is at most
  \(c_0(e^{-W}+\log n/n^{1/4})X^{3}/M^{3}\).

  In the main range \(Xe^{-W} < p \le X\) put \(a = \log(X/p) \in [0,W)\). By
  \eqref{eq:pnt-sub},
  \begin{align*}
    \bigl(\pi(p)-1\bigr)\pi(n/p)
      =\frac{X^{2}}{M^{2}}
       &\Bigl(e^{-a}+\theta_1\tfrac{A_1(W+1)}{M}-\tfrac{M}{X}\Bigr)\\
       &\times
       \Bigl(e^{a}+\theta_2\tfrac{A_1(W+1)e^{W}}{M}\Bigr).
  \end{align*}
  Expanding, the leading term is \(e^{-a}e^{a} = 1\), and every other term
  carries a factor \(A_1(W+1)e^{W}/M\) or \(M/X\); using \(e^{-a} \le 1\),
  \(e^{a} \le e^{W}\) and \(X \ge M^{2}\), all of them together are at most
  \(A_2(W+1)e^{W}/M\) in absolute value, with \(A_2\) absolute. Hence
  \begin{align*}
    \sum_{Xe^{-W} < p \le X}\bigl(\pi(p)-1\bigr)\pi(n/p)
      ={}&\frac{X^{2}}{M^{2}}
          \Bigl(1+\theta_3\,\tfrac{A_2(W+1)e^{W}}{M}\Bigr)\\
        &\times \#\{\,p : Xe^{-W} < p \le X\,\},
  \end{align*}
  with \(|\theta_3|\le 1\). This is where the summand being constant to leading
  order does the work: no partition of \([0,W]\) and no Riemann sum is needed,
  only the number of primes in the range. That number is, by
  \eqref{eq:pnt-sub} again,
  \[
    \pi(X)-\pi\bigl(Xe^{-W}\bigr)
      =\frac{X}{M}\Bigl(1-e^{-W}+\theta_4\,\tfrac{2A_1(W+1)}{M}\Bigr).
  \]
  Multiplying, and adding the tail bound,
  \[
    \frac{M^{3}}{X^{3}}T(n)
      =1+O\!\left(e^{-W}+\frac{(W+1)e^{W}}{M}+\frac{\log n}{n^{1/4}}\right),
  \]
  with an absolute implied constant. Now take \(W = \tfrac{1}{2}\log M\), legitimate
  for \(M \ge e^{2}\) since then \(W \ge 1\) and \(M \ge 2W\). It gives
  \[
    e^{-W} = M^{-1/2},\qquad \frac{(W+1)e^{W}}{M} = \frac{\tfrac{1}{2}\log M+1}{\sqrt M},
  \]
  both \(O(\log M/\sqrt M)\), while \(\log n/n^{1/4}\) is smaller than either.
\end{proof}

\begin{theorem}\label{thm:error}
  There are absolute constants \(C\) and \(n_0\) such that for all \(n \ge n_0\)
  \begin{equation}\label{eq:explicit}
    \left|\;\frac{S_2(n)}{\pi(\sqrt n)^{3}}-\frac{2}{3}\;\right|
      \ \le \ C\,\frac{\log\log n}{\sqrt{\log n}} .
  \end{equation}
  In particular, as \(n \to \infty\),
  \begin{equation}\label{eq:Sasymp}
    S(n)\;\sim \;\tfrac{2}{3}\,\pi(\sqrt n)^{3}
            \;\sim \;\frac{16}{3} \cdot \frac{n^{3/2}}{\log^{3}n},
  \end{equation}
  and consequently
  \begin{equation}\label{eq:Cnexpand}
    C_n = \binom{\pi(n)+1}{2}
        +\Bigl(\tfrac{16}{3}+o(1)\Bigr)\frac{n^{3/2}}{\log^{3}n}.
  \end{equation}
  In particular the upper bound of Theorem~\ref{thm:asymp} overshoots by a
  factor \(\asymp\log^{2}n\).
\end{theorem}

\begin{proof}
  By \eqref{eq:sqexact} it suffices to bound \(T(n)/P^{3}-1\), since the
  remaining terms \(1/(2P)-1/(6P^{2})\) are \(O(1/P) = O(\log n/\sqrt n)\), far
  smaller than the claimed bound. By Lemma~\ref{lem:pnt-unif} with \(a = 0\) and
  \(W = 1\),
  \[
    P = \pi(X) = (X/M)(1+2\theta A_1/M),
  \]
  so
  \[
    \frac{T(n)}{P^{3}} = \frac{M^{3}}{X^{3}}T(n) \cdot \Bigl(\frac{X}{MP}\Bigr)^{3}
      =\Bigl(1+O\bigl(\tfrac{\log M}{\sqrt M}\bigr)\Bigr)
       \Bigl(1+O\bigl(\tfrac{1}{M}\bigr)\Bigr)
      =1+O\!\left(\frac{\log M}{\sqrt M}\right)
  \]
  by Proposition~\ref{prop:T}. Substituting \(M = \tfrac{1}{2}\log n\) gives
  \eqref{eq:explicit}.

  For the asymptotic statements, \(P = \pi(\sqrt n)\) and, by the prime number
  theorem,
  \[
    \pi(\sqrt n)^{3} \sim \Bigl(\frac{2\sqrt n}{\log n}\Bigr)^{3} = \frac{8n^{3/2}}{\log^{3}n},
  \]
  so
  \[
    S_2(n) \sim \tfrac{2}{3}\pi(\sqrt n)^{3} \sim \frac{16}{3}\cdot\frac{n^{3/2}}{\log^{3}n}.
  \]
  Lemma~\ref{lem:semiprime} transfers this to \(S\), and Corollary
  \ref{cor:total} to \(C_n\).
\end{proof}

\begin{remark}\label{rem:hypotheses}
  Five comments on the proof, the first two of which are traps.

  \begin{enumerate}
    \item \emph{No sieve theory is needed, and reaching for it gives the wrong
      answer.} The mass sits where \(p\) and \(n/p\) are both within a bounded
      factor of \(\sqrt n\), that is, at \(u = \log(n/p)/\log p \to 1\); and there
      \(\Phi\) degenerates to the exact identity \(\Phi(x,y) = \pi(x)-\pi(y)+1\),
      valid for \(y \ge \sqrt x\). The one-term Buchstab estimate
      \(\Phi(x,y) \approx x\,\omega(u)/\log y\), with \(\omega\) Buchstab's
      function, fed into \eqref{eq:Sdef}, returns \(8\) rather than \(16/3\);
      and it is uniform only for \(u \ge 2\) \cite{Fan}, which excludes precisely
      the region carrying the mass. Lemma~\ref{lem:split} says exactly what goes
      missing. On \(1 \le u \le 2\) one has \(\omega(u) = 1/u\), so the one-term
      estimate amounts to replacing the second factor of \(g\) by \(\pi(n/p)\)
      alone, that is, to dropping the \(-\pi(p)+1\). What is dropped is
      therefore
      \[
        \sum_j(j-1)^{2} = \frac{(P-1)P(2P-1)}{6} \sim \frac{P^{3}}{3},
      \]
      which is
      \(\tfrac{8}{3}\,n^{3/2}/\log^{3}n\); restoring it gives
      \(8-\tfrac{8}{3} = \tfrac{16}3\). So the discrepancy is not a subtlety about
      \(\Phi\) at all, but the exact square sum that Lemma~\ref{lem:split}
      peels off.

    \item \emph{The convergence is slow and not monotone}, so a numerical check
      over too short a range is misleading. The raw quotient
      \(S(n)\log^{3}n/n^{3/2}\) peaks near \(10.9\) around \(n = 7 \cdot 10^{4}\)
      and is still \(9.55\) at \(n = 10^{7}\), against a limit of
      \(16/3 = 5.33\dots\). Most of that discrepancy, about two thirds of it on a
      logarithmic scale over the range tabulated, is the difference between
      \(\pi(\sqrt n)\) and \(2\sqrt n/\log n\), which is cubed here; the factor
      \(\bigl(\pi(\sqrt n)\log n/2\sqrt n\bigr)^{3}\) is \(1+6/\log n\) to first
      order, and is \(1.56\) at \(n = 10^{6}\). Measured instead against
      \(\tfrac{2}{3}\pi(\sqrt n)^{3}\), which absorbs that factor entirely, the
      picture is much better behaved, though still not monotone: see Figure
      \ref{fig:ratio}.

    \item \emph{The proof uses no Riemann sums, no partition of the range, and
      no interchange of limits}: \(W\) is an explicit function of \(n\), and
      \eqref{eq:explicit} is a single estimate. That is a consequence of Lemma
      \ref{lem:split}, which removes the \(\tfrac{1}{3}\) exactly, together with the
      observation that what remains has a summand constant to leading order.

    \item \emph{The rate is not sharp.} It is limited by the truncation at
      \(a = W\), not by the arithmetic input: at \(n = 10^{14}\) the left-hand side
      of \eqref{eq:explicit} is \(0.011\), against \(0.61\) for
      \(\log\log n/\sqrt{\log n}\). How much better the truth is we do not
      claim; multiplying the computed deviations by \(\log n\) gives
      \(0.55,0.92,0.59,0.68,0.54,0.47,0.37\) at \(n = 10^{6},\dots,10^{14}\),
      which neither settles nor rules out a rate of \(1/\log n\). Sharpening the
      theorem would mean treating the whole range \(0 \le a \le M\) rather than
      discarding \(a > W\).

    \item \emph{The hypotheses can be weakened} if one wants only the asymptotic
      and no rate. Lemma~\ref{lem:pnt-unif} was stated with the prime number
      theorem in the sharp form \(\pi(y) = y/\log y+O(y/\log^{2}y)\) because that
      is what produces an explicit error. The bare prime number theorem
      \(\pi(y) \sim y/\log y\) already suffices for
      \(S(n) \sim \tfrac{2}{3}\pi(\sqrt n)^{3}\): for \(a\) in a \emph{fixed} range
      \([0,W]\) the arguments \(Xe^{\pm a}\) tend to infinity with \(n\), so
      \eqref{eq:pnt-sub} holds with the error terms replaced by an unspecified
      \(\varepsilon_W(n) \to 0\), uniformly in \(a\); one then lets \(W \to \infty\)
      afterwards, the tail of Lemma~\ref{lem:tail} being \(O(e^{-W})\). This
      matters to anyone wishing to formalise the argument in a proof assistant,
      since it removes the need for a prime number theorem \emph{with} an error
      term.\remarkendhere
  \end{enumerate}
\end{remark}

\begin{figure}[ht]
  \centering
  \includegraphics[width=\linewidth]{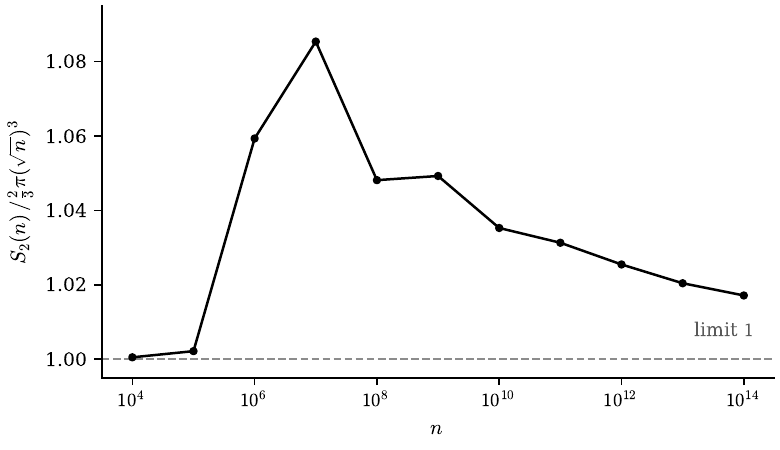}
  \caption{The quotient \(S_2(n)\big/\tfrac{2}{3}\pi(\sqrt n)^{3}\) against \(n\),
    on a logarithmic \(n\)-axis, for \(n = 10^{4},\dots,10^{14}\). Theorem
    \ref{thm:error} asserts that this tends to \(1\). It does, but only after
    rising to a maximum of \(1.085\) at \(n = 10^{7}\), and with a further slight
    rise between \(10^{8}\) and \(10^{9}\); at \(n = 10^{14}\) it is still
    \(1.017\). The two smallest values, at \(10^{4}\) and \(10^{5}\), sit
    within \(0.3\%\) of \(1\) by coincidence rather than for any reason, which
    is a useful warning on its own. Values for \(n \le 10^{9}\) are from a
    least-prime-factor sieve, those beyond from an \(O(n^{3/4})\)
    prime-counting recursion \cite{gitrepo}.}\label{fig:ratio}
\end{figure}

\begin{table}[ht]
  \centering
  \begin{tabular}{rrrr}
    \toprule
    \(n\) & \(C_n\) & \(\binom{\pi(n)+1}{2}\) & ratio \\
    \midrule
    \(10^{2}\) & \(362\) & \(325\) & \(1.1138\) \\
    \(10^{3}\) & \(14\,957\) & \(14\,196\) & \(1.0536\) \\
    \(10^{4}\) & \(769\,090\) & \(755\,835\) & \(1.0175\) \\
    \(10^{5}\) & \(46\,233\,662\) & \(46\,008\,028\) & \(1.0049\) \\
    \(10^{6}\) & \(3\,084\,943\,710\) & \(3\,081\,007\,251\) & \(1.0013\) \\
    \(10^{7}\) & \(220\,905\,096\,960\) & \(220\,832\,955\,910\) & \(1.00033\) \\
    \bottomrule
  \end{tabular}
  \caption{\(C_n\) against the lower bound of Theorem 3.4(3) of
    \cite{Snellman2000}.}\label{tab:err}
\end{table}

\section{The socle, and the type}\label{sec:socle}

This section is routine: every ingredient is standard, and for \(I_n\) the
evaluation is a one-line application. We include it because the answer is a
quantity that has already appeared twice above, and because the contrast with
the unitary case in Remark~\ref{rem:unitary} seems worth recording.

Recall the standard multigraded reduction: for a monomial ideal \(I \subsetneq S\),
the socle \((0:_{S/I}\mathfrak m)\) is spanned by the residues of those
monomials \(u \notin I\) with \(x_iu \in I\) for every \(i\). This holds because
\(S/I\) is \(\ZZ^{r}\)-graded with one-dimensional components, so the
annihilator is a multigraded subspace and hence spanned by the monomials in it. See
\cite{HerzogHibi} or \cite{MillerSturmfels}.

\begin{proposition}\label{prop:socle}
  Let \(n \ge 2\). Then \(\operatorname{soc}(\Gam_n)\) is spanned by the
  monomials of weight \(m\) with \(n/2 < m \le n\); under the identification of
  the \(K\)-basis of \(\Gam_n\) with \(\{1,\dots,n\}\),
  \begin{equation}\label{eq:socle}
    \operatorname{soc}(\Gam_n) = \bigoplus_{n/2 < m \le n}K \cdot m,
    \qquad
    \dim_K\operatorname{soc}(\Gam_n) = \Bigl\lceil\frac{n}{2}\Bigr\rceil .
  \end{equation}
\end{proposition}

\begin{proof}
  By the reduction above, \(u\) is a socle monomial exactly when \(w(u) \le n\)
  and \(w(x_iu) = p_iw(u) > n\) for every \(i = 1,\dots,r\). Since \(p_1 = 2 \le p_i\)
  for all \(i\), the condition at \(i = 1\) implies all the others, so the system
  reduces to that single binding constraint: \(2w(u) > n\). So the socle monomials are those of weight
  in \((n/2,n]\), and there are \(n-\lfloor n/2\rfloor = \lceil n/2\rceil\)
  of them.
\end{proof}

\begin{corollary}\label{cor:type}
  For \(n \ge 2\) the Cohen--Macaulay type of \(\Gam_n\) is
  \begin{equation}\label{eq:type}
    \operatorname{type}(\Gam_n) = \beta^{S}_{r}(\Gam_n)
      =C_{n,1} = \Phi(n,2) = \Bigl\lceil\frac{n}{2}\Bigr\rceil,
  \end{equation}
  and \(\Gam_n\) is Gorenstein if and only if \(n \le 2\).
\end{corollary}

\begin{proof}
  For an Artinian module the last Betti number over \(S\) is the socle
  dimension, so the first two equalities are Proposition~\ref{prop:socle}. For
  the third, Eliahou--Kervaire gives
  \(\beta^{S}_{i} = \sum_{v}C_{n,v}\binom{r-v}{i-1}\), and at \(i = r\) the
  binomial vanishes unless \(v = 1\), leaving \(\beta^{S}_{r} = C_{n,1}\); the last
  equality is Theorem~\ref{thm:main}. Gorenstein means type \(1\), i.e.
  \(\lceil n/2\rceil = 1\), i.e.\ \(n \le 2\).
\end{proof}

\begin{remark}
  Corollary~\ref{cor:type} is not an independent third sighting of
  \(\lceil n/2\rceil\), and it would be misleading to present it as one.
  For an Artinian Golod quotient the denominator of \eqref{eq:PB} is the list
  of Betti numbers,
  \[
    D_n(t) = 1-\sum_{i \ge 1}\beta^{S}_{i}\,t^{\,i+1};
  \]
  since \(\operatorname{pd}_S\Gam_n = r\), its leading coefficient is
  \(-\beta^{S}_{r} = -\operatorname{type}(\Gam_n)\), and dividing by the monic
  \((1+t)^{\,r-\ell_1}\) preserves it. So clause (6) of Conjecture 4.6
  \emph{is} the assertion that the type is \(\lceil n/2\rceil\), and the proof
  of clause (6) given above, where the top-degree term arises only from
  \(j = 1\), is the Eliahou--Kervaire computation \(\beta^{S}_{r} = C_{n,1}\)
  in generating-function clothing. What does remain is a coincidence of
  two countings: the socle counts the integers in \((n/2,n]\), while
  \(\Phi(n,2)\) counts the odd integers in \([1,n]\), and these are different sets
  of the same size.
\end{remark}

\begin{remark}\label{rem:unitary}
  The reason Proposition~\ref{prop:socle} is a one-liner is that the socle
  condition is ``\(p_iw(u) > n\) for \emph{every} prime \(p_i \le n\)'', which
  collapses because the least prime is \(2\). The analogue for the
  \emph{unitary} truncations does not collapse: there the condition reads
  ``\(kp > n\) for every prime \(p\) coprime to \(k\)'', and
  \cite[\S7]{SnellmanUnitaryN} obtains the socle dimension only
  asymptotically, as \(\delta n\) with
  \[
    \delta = \frac{1}{2}+\sum_{i \ge 1}
      \frac{p_i^{-1}-p_{i+1}^{-1}}{p_1p_2\cdots p_i} \approx 0.6077 .
  \]
  For Dirichlet convolution nothing of the sort is needed: the socle is
  exactly the integers in \((n/2,n]\), for every \(n\), so its density is
  \(1/2\) in the limit and \(1/2+O(1/n)\) at every finite stage.
\end{remark}

\begin{remark}
  Grading by \(\lambda = \Om\), the algebra \(\Gam_n\) is \emph{level}, that is,
  has its socle concentrated in one degree, exactly for \(n \le 3\). Indeed for \(n \ge 4\)
  let \(2^{k}\) be the largest power of \(2\) with \(2^{k} \le n\); then
  \(2^{k} > n/2\), so \(2^{k}\) is a socle element of degree \(k \ge 2\), while
  Bertrand's postulate applied to \(\lfloor n/2\rfloor\) supplies a prime
  \(q\) with \(n/2 < q \le n\), a socle element of degree \(1\).
\end{remark}

\section{Related work}\label{sec:related}

First, on the algebraic side, the ring \(\Gam\) itself continues to be studied, for instance recently in \cite{GoswamiKleynPorrill}, which shows it is neither
Noetherian nor Artinian, constructs several families of prime ideals, and
computes its Krull dimension to be infinite. Its finite-dimensional truncations
appear not to have been taken up by anyone else, but they were taken up once
more by the present author, from a different direction: \(\{1,2,\dots,n\}\) is
a \emph{shifted multicomplex}, the complement of a strongly stable ideal being
to a multicomplex what a shifted simplicial complex is to a squarefree one, and
\cite{SnellmanLaplacians} accordingly computes the spectrum of its Laplacian.
By the theorem of Duval and Reiner \cite{DuvalReiner} that spectrum is integral,
and \cite[Cor.~14]{SnellmanLaplacians} identifies its multiplicities as the
arithmetic counts
\[
  t_{i,k}(n) = \#\{\,1 < m \le n : \Om(m) = k,\ p_i \mid \operatorname{sfp}(m)\,\},
\]
\(\operatorname{sfp}\) denoting the squarefree part. These are not the counts
\(C_{n,v}\) of the present note: they fix the total degree and impose a
divisibility condition, rather than fixing the least support. So neither
result subsumes the other, and it is a natural question, and the exact analogue
of Theorem~\ref{thm:main}, whether \(t_{i,k}\) too admits a sieve-theoretic
closed form. Artinian truncations of Stanley--Reisner rings are an active topic
\cite{Miyashita}, and \(\Gam_n\) is an object of that general shape, but no
arithmetic enters there.

Second, and more to the point, the bridge Theorem~\ref{thm:main} builds appears
to be new. The asymptotic behaviour of \(\Phi(x,y)\) is classical: Buchstab
\cite{Buchstab} obtained
\[
  \Phi(x,x^{1/u}) \sim u\,\omega(u)\,\frac{x}{\log x}
\]
in terms of the function that now bears his name, and de Bruijn \cite{deBruijn}
gave an approximation uniform for all \(x \ge y \ge 2\); see
\cite{Tenenbaum} for a textbook account, \cite{Fan} for numerically explicit
estimates, \cite{Gorodetsky} for short intervals, and \cite{Holt} for the
regime of fixed \(p\), which is the one closest to ours. We have not been able
to find anywhere in the literature a connection between these counts and any
invariant of a free resolution; the two subjects appear not to cite each other
at all. (A search for ``Buchstab'' in commutative algebra is apt to return
instead the bigraded Betti numbers of face rings studied by V.~M.~Buchstaber,
which is an unrelated subject and an unrelated person.)

\section{What is next}\label{sec:next}

Theorem~\ref{thm:main} moves the combinatorics of \(G(I_n)\) wholesale into
sieve theory, and the questions it raises are correspondingly analytic.

\begin{openproblem}
  Give asymptotics for the whole profile \(v\mapsto C_{n,v} = \Phi(n,p_v)\),
  \emph{uniformly} in \(v\), and deduce asymptotics for the graded Betti
  numbers of \(\Gam_n\) through Corollary 4.2 of \cite{Snellman2000}. The
  behaviour of \(\Phi(x,y)\) is classical \cite{Buchstab,deBruijn,Tenenbaum},
  but here \(y = p_v\) sweeps the whole interval \([2,n]\) as \(v\) does, so what
  is wanted is a statement uniform across all of it, including the two
  degenerate ends \(v = 1\), where \(\Phi(n,2) = \lceil n/2\rceil\), and
  \(v = \pi(n)\), where \(\Phi(n,p_{r}) = 1\). The uniform approximation of
  \cite{deBruijn} and the explicit estimates of \cite{Fan} are the natural
  starting points.
\end{openproblem}

\begin{openproblem}
  Theorem~\ref{thm:main}\({}'\) determines the refined counts \(C_{n,v,d}\)
  exactly, through the generating function \(\Phi(x,y;u)\), and the remark
  following it shows that the restriction to a fixed number of prime factors
  alone does not. What it does not give is their size: find
  asymptotics for \(C_{n,v,d}\), uniformly in \(v\) and \(d\), and thence for
  the graded Betti numbers of \(\Gam_n\) in each internal degree. The
  \(u\)-analogue \(\Phi(x,y;u)\) is the natural object, since at \(u = 1\) it
  returns the ungraded profile of the previous problem.
\end{openproblem}

\begin{openproblem}
  Dirichlet convolution is the coarsest of Narkiewicz's regular convolutions.
  Carry out the analysis of \cite{Snellman2000} and of this note for the
  unitary, ternary, and general greedy convolutions: which truncation ideals are
  stable, and does an analogue of Theorem~\ref{thm:main} hold? For the unitary
  case the truncations were treated by Stanley--Reisner methods in
  \cite{SnellmanUnitaryGeneral,SnellmanUnitaryN}, and it is not clear that a
  sifting function
  appears there at all. This is work in progress by the author, who intends to
  treat the unitary case next.
\end{openproblem}

\begin{openproblem}\label{op:lean}
  The elementary half of this note involves no homological algebra and is
  within reach of a proof assistant. Part of it has been formalised in Lean~4
  and Mathlib, and is available in \cite{gitrepo}: Corollary~\ref{cor:split}
  and the whole polynomial argument of Section~\ref{sec:conjecture}
  (Lemma~\ref{lem:secondiff}, Proposition~\ref{prop:order} and clauses
  (1), (2), (4), (5) of Theorem~\ref{thm:conj}) are proved outright, with no
  \texttt{sorry} and on the standard axioms. What remains is (i) clauses (3)
  and (6), which are degree and leading-coefficient computations, fiddly rather
  than deep; (ii) Propositions~\ref{prop:fixed344}--\ref{prop:fixed345}; and
  (iii) Theorems~\ref{thm:S32} and \ref{thm:main} themselves. Item (iii) is the
  real obstruction, and it is the same one that puts Section
  \ref{sec:conjecture}'s homological input out of reach: the formalisation
  takes \(C_{n,v} = \Phi(n,p_v)\) as a definition and checks it numerically,
  because \(G(I_n)\) cannot even be stated without monomial ideals, which
  Mathlib does not have, any more than it has minimal free resolutions,
  graded Betti numbers or Golod rings.
\end{openproblem}

Finally, three integer sequences arising here deserve to be in the OEIS \cite{OEIS}.
\(C_n\) is \seqnum{A182843}, whose entry records none of the above; the
sequence \(C_{n,2}\) and the triangle \(C_{n,v}\) appear not to be present at
all.

\section*{Notation}

Collected for reference; equation numbers point at the defining display.

\begingroup\small
\begin{center}
\begin{tabular}{@{}ll@{}}
  \toprule
  \multicolumn{2}{@{}l}{\emph{Arithmetical}}\\
  \(p_1 = 2 < p_2 < \cdots\)      & the primes \\
  \(\pi(x)\)                & the number of primes \(\le x\) \\
  \(\Om(m)\), also \(\lambda\) & number of prime factors of \(m\), with multiplicity \\
  \(\lpf(m)\)               & least prime factor of \(m > 1\) \\
  \(\operatorname{sfp}(m)\) & squarefree part of \(m\) \\
  \(\Phi(x,y)\)             & Legendre's sifting function, \eqref{eq:phidef} \\
  \(\Phi(x,y;u)\), \(\Phi_{k}(x,y)\) & the same weighted by \(u^{\Om}\), \eqref{eq:phiu}; restricted to \(\Om = k\) \\
  \(R_p(y)\)                & number of \(p\)-rough integers \(\le y\) \\
  \(\omega(u)\)             & Buchstab's function, \emph{not} the count of distinct prime factors \\
  \addlinespace
  \multicolumn{2}{@{}l}{\emph{Of this note and of \cite{Snellman2000}}}\\
  \(\Gam\), \(\Gam_n\)      & the ring of number-theoretic functions, and its truncation \\
  \(S\), \(I_n\)            & \(K[x_1,\dots,x_r]\) with \(r = \pi(n)\), and the truncation ideal \\
  \(w(m)\)                  & weight of a monomial, \(\prod_ip_i^{a_i}\) \\
  \(G(I_n)\)                & the minimal monomial generating set \\
  \(\min(m)\), \(\max(m)\)  & least and greatest index occurring in \(m\) \\
  \(\operatorname{supp}(m)\) & the set of indices occurring in \(m\) \\
  \(C_{n,v}\), \(C_n\)      & generator counts by least support, and their total, \eqref{eq:Cdef} \\
  \(C_{n,v,d}\)            & the same, by least support and degree, \eqref{eq:Cvddef} \\
  \(E(n,v)\)                & composites \(\le n\) with every prime factor \(>p_v\), \eqref{eq:Edef} \\
  \(D_n\), \(q_n\)          & denominators of the Poincaré--Betti series, \eqref{eq:PB} \\
  \(\ell_1\), \(\ell_2\)    & the exponents of Conjecture 4.6 \\
  \(S(n)\), \(S_2(n)\)      & the error term \eqref{eq:Sdef}, and its \(\Om = 2\) part \\
  \(T(n)\), \(P\)           & Lemma~\ref{lem:split}; \(P = \pi(\sqrt n)\) \\
  \(X\), \(M\)              & \(\sqrt n\) and \(\log\sqrt n\), \eqref{eq:XM} \\
  \(\operatorname{soc}\)    & socle, Section~\ref{sec:socle} \\
  \bottomrule
\end{tabular}
\end{center}
\endgroup

Four collisions are worth flagging. Against standard usage: \(\omega\) here is
Buchstab's function, not the number of \emph{distinct} prime factors, which this
note never uses. Against the accompanying computational report \cite{gitrepo}:
\(T(n)\) is the sum of Lemma~\ref{lem:split}, while that report uses the same
letter for an unrelated quantity in its \(\Om \ge 3\) table. And two internal ones,
both harmless but worth naming: \(S\) is the polynomial ring while \(S(n)\) is
the error term \eqref{eq:Sdef}, and \(q\) is a generic prime in Sections
\ref{sec:main}--\ref{sec:errata} while \(q_n\) is a polynomial in Section
\ref{sec:conjecture}.

\section*{Acknowledgements}

The whole subject of this note, and of \cite{Snellman2000} before it, is owed
to \textbf{Johan Andersson} \orcidlink{0000-0002-9651-1766}. It was he who
suggested studying the truncations \(\Gam_n\) in the first place, and he who
pointed out that the ideals \(I_n\) are monomial, and that this ought,
therefore, to be my cup of tea. It was, and the observation has now paid out
twice.

His hope at the time was specific, and it is worth recording because a quarter
of a century has now furnished a partial answer to it. Monomial ideals carry a
great deal of theory; the truncations \(\Gam_n\) are monomial and arithmetic at
once; so, he reasoned, the algebraic machinery ought to run backwards and
deliver number-theoretic conclusions for nothing. That is not what has happened.
The traffic has gone almost entirely the other way. It is the algebra that has
supplied the questions (Conjecture 4.6, which is in the end the assertion
that \(v\mapsto C_{n,v}\) is exactly linear beyond \(v = \ell_1(n)\), and the error
term in Theorem~\ref{thm:asymp}, which is a Mertens-type sum), and answering
them as number theory is what has illuminated the algebra. Theorem
\ref{thm:error} is the clearest case: a statement about the minimal free
resolution of \(\Gam_n\), settled by an argument containing no algebra at all.

There is, I think, a structural reason for the one-way traffic. Most of what commutative algebra knows about
a monomial ideal is \emph{extremal}: Macaulay's theorem, Kruskal--Katona, the
Bigatti--Hulett--Pardue bound. Each of these constrains an arbitrary object with
a given Hilbert function, and is therefore only as strong as the worst such
object. But \(\Gam_n\) is nowhere near extremal, and the resulting bounds are
correspondingly slack. Concretely: \(\{1,2,\dots,n\}\) is a multicomplex, so by
Macaulay's theorem the sequence
\[
  d\ \longmapsto\ \#\{\,m \le n : \Om(m) = d\,\},
\]
the Landau counts, must be an \(M\)-sequence, which yields growth bounds free of
any analytic input. They are true and they are useless. At \(n = 200\) there are
\(46\) primes, and Macaulay permits as many as \(1081\) integers with
\(\Om = 2\); the true number is \(62\). No inequality of that kind is going to
tell an analytic number theorist something he does not already know.

Where the transfer does work is where the algebra is \emph{exact} rather than
extremal. Eliahou--Kervaire gives the graded Betti numbers on the nose, and it
is precisely that exactness which makes this note possible at all: the identity
\(C_{n,v} = \Phi(n,p_v)\) is a statement about an invariant that the algebra
computes rather than bounds. Section~\ref{sec:socle} is a small instance of the
same phenomenon, and \cite{SnellmanLaplacians}, where Duval--Reiner
integrality forces the Laplacian spectrum of \(\Gam_n\) to be a list of
arithmetic counting functions, is another. So Johan's hope was not wrong so
much as mis-aimed. The fruit is there; it is just not low-hanging, and it does
not grow on the general-inequality branch of the tree.

\begin{center}\small\(\ast\qquad\ast\qquad\ast\)\end{center}

Just as Geheimrat Goethe had the diligent Friedrich Wilhelm Riemer to help him
write \emph{Die Wahlverwandtschaften}, just as Grothendieck had Jean Dieudonné
to help him complete EGA, I have \textbf{Claude}! What exactly that amounted to
is set out in the disclosure section below.

\begin{center}\small\(\ast\qquad\ast\qquad\ast\)\end{center}

It is a pleasure to acknowledge the computer algebra systems without which
none of this would exist. The computations behind \cite{Snellman2000} were
carried out in \textsf{Maple}: the generator counts of its Figures 1 and 2,
and the Poincaré--Betti series of its Figure 3, assembled from those counts
through the formulas of its Corollaries 4.2 and 4.4. In the companion papers
on unitary convolution the division of labour was quite different. For
\cite{SnellmanUnitaryGeneral,SnellmanUnitaryN}, \textsf{Maple} was used to \emph{generate}
\textsf{Macaulay2} \cite{Macaulay2} input (procedures emitting ring and
ideal declarations line by line), and \textsf{Macaulay2} then computed the
free resolutions, Betti numbers and Poincaré--Betti series. The homology of the
relevant simplicial complex was computed with the \textsf{Simplicial Homology}
package \cite{SimplicialHomology} of Dumas, Heckenbach, Saunders and Welker for
\textsf{GAP} \cite{GAP}; having found it torsion-free, the author wrote a short
\textsf{GAP} programme to test lex-shellability, which it duly was. Every
computation in the present note was done in \textsf{SageMath} \cite{Sage}, and
deliberately along a route different from the one taken
in 1999: the minimal generators of \(I_n\) are enumerated as integers rather
than counted by the formulas under test, so that agreement between the two is
evidence rather than tautology.

\section*{Disclosure of AI assistance}

This manuscript was produced with substantial assistance from a large language
model, Claude (Anthropic). Its extent is set out below.

The errata of Section~\ref{sec:errata} were found by Claude while re-reading
\cite{Snellman2000} with the author, by recomputing every table in that paper
along a route independent of the one used there. Found likewise by Claude,
working with the author in interactive sessions: Theorem~\ref{thm:main} and its
proof, and its graded refinement, Theorem~\ref{thm:main}\({}'\); the reduction of Conjecture 4.6 to the order of vanishing of \(D_n\) at
\(t = -1\) (Lemma~\ref{lem:secondiff} and Proposition~\ref{prop:order}); the
asymptotic Theorem~\ref{thm:asymp}; the identification of \(C_n\) with
\seqnum{A182843}; Theorem~\ref{thm:error} together with Lemmas
\ref{lem:semiprime}, \ref{lem:pnt-unif} and \ref{lem:tail}, and the two
cautions now folded into Remark~\ref{rem:hypotheses}; and
the contents of Section~\ref{sec:socle}. Claude also located and verified the
references, drew Figures~\ref{fig:profile} and \ref{fig:ratio}, wrote the accompanying
\textsf{SageMath} code and the Lean formalisation mentioned in Open Problem
\ref{op:lean}, and drafted this manuscript, with the author directing the
exposition throughout.

Every numerical claim in the note was checked along a route independent of the
formula being tested, and the reference list has been checked against primary
sources. The author has verified the results to the best of his ability, made
the final edits, and assumes full responsibility for any errors, mistakes or
gaps that remain.

\section*{Funding}

This research received no external funding. It was carried out as part of the
author's regular duties at Linköping University, within the fraction of that
employment allocated to research.

\section*{Conflicts of interest}

The author declares no financial or non-financial conflicts of interest. The
nature and extent of large-language-model assistance in producing this
manuscript is disclosed in full in the section above.


\begin{thebibliography}{Sne02a}

\bibitem[AH96]{AramovaHerzog}
Annetta Aramova and Jürgen Herzog,
\emph{Koszul cycles and Eliahou--Kervaire type resolutions},
J. Algebra \textbf{181} (1996), no.~2, 347--370.

\bibitem[Buc37]{Buchstab}
A.~A. Buchstab,
\emph{Asymptotic estimates of a general number-theoretic function},
Mat. Sb. \textbf{44} (1937), 1239--1246. (Russian.)

\bibitem[CE59]{CashwellEverett}
E.~D. Cashwell and C.~J. Everett,
\emph{The ring of number-theoretic functions},
Pacific J. Math. \textbf{9} (1959), 975--985.

\bibitem[dB50]{deBruijn}
N.~G. de Bruijn,
\emph{On the number of uncancelled elements in the sieve of Eratosthenes},
Indag. Math. \textbf{12} (1950), 247--256.

\bibitem[DHSW]{SimplicialHomology}
Jean-Guillaume Dumas, Frank Heckenbach, B.~David Saunders and Volkmar Welker,
\emph{Simplicial Homology}, a \textsf{GAP} package.

\bibitem[DR02]{DuvalReiner}
Art M. Duval and Victor Reiner,
\emph{Shifted simplicial complexes are Laplacian integral},
Trans. Amer. Math. Soc. \textbf{354} (2002), no.~11, 4313--4344.

\bibitem[EK90]{EliahouKervaire}
Shalom Eliahou and Michel Kervaire,
\emph{Minimal resolutions of some monomial ideals},
J. Algebra \textbf{129} (1990), no.~1, 1--25.

\bibitem[Fan23]{Fan}
Steve Fan,
\emph{Numerically explicit estimates for the distribution of rough numbers},
arXiv:2306.03347.

\bibitem[GAP]{GAP}
The GAP Group,
\emph{GAP --- Groups, Algorithms, and Programming},
\url{https://www.gap-system.org}.

\bibitem[GKP24]{GoswamiKleynPorrill}
Amit Goswami, Steven Kleyn and Lauren Porrill,
\emph{On structures of the ring of arithmetical functions: prime ideals and
beyond}, arXiv:2410.10824.

\bibitem[GL69]{GulliksenLevin}
Tor H. Gulliksen and Gerson Levin,
\emph{Homology of local rings},
Queen's Papers in Pure and Applied Mathematics, vol.~20, Queen's University,
Kingston, Ont., 1969.

\bibitem[Gol62]{Golod}
E.~S. Golod,
\emph{On the homology of some local rings},
Soviet Math. Dokl. \textbf{3} (1962), 745--749.

\bibitem[Gor24]{Gorodetsky}
Ofir Gorodetsky,
\emph{The variance of integers without small prime factors in short intervals},
Math. Z. \textbf{308} (2024), no.~4, Paper No.~59; arXiv:2111.00853.

\bibitem[GS]{Macaulay2}
Daniel R. Grayson and Michael E. Stillman,
\emph{Macaulay2, a software system for research in algebraic geometry},
\url{https://macaulay2.com}.

\bibitem[HH11]{HerzogHibi}
Jürgen Herzog and Takayuki Hibi,
\emph{Monomial Ideals},
Grad. Texts in Math., vol.~260, Springer, 2011.

\bibitem[Hol23]{Holt}
Fred B. Holt,
\emph{On the counts of \(p\)-rough numbers},
arXiv:2308.07570.

\bibitem[HRW99]{HerzogReinerWelker}
Jürgen Herzog, Victor Reiner and Volkmar Welker,
\emph{Componentwise linear ideals and Golod rings},
Michigan Math. J. \textbf{46} (1999), no.~2, 211--223.

\bibitem[Miy26]{Miyashita}
Sora Miyashita,
\emph{Canonical traces of Artinian truncations of Stanley--Reisner rings},
arXiv:2608.15955.

\bibitem[MP15]{McCulloughPeeva}
Jason McCullough and Irena Peeva,
\emph{Infinite graded free resolutions},
in: Commutative Algebra and Noncommutative Algebraic Geometry, Vol.~I,
Math. Sci. Res. Inst. Publ. \textbf{67}, Cambridge Univ. Press, 2015,
pp.~215--257.

\bibitem[MS05]{MillerSturmfels}
Ezra Miller and Bernd Sturmfels,
\emph{Combinatorial Commutative Algebra},
Grad. Texts in Math., vol.~227, Springer, 2005.

\bibitem[OEIS]{OEIS}
The On-Line Encyclopedia of Integer Sequences, published electronically at
\url{https://oeis.org}, sequence \seqnum{A182843}.

\bibitem[Pee96]{Peeva}
Irena Peeva,
\emph{0-Borel fixed ideals},
J. Algebra \textbf{184} (1996), no.~3, 945--984.

\bibitem[Rep]{gitrepo}
Supporting code, raw computational output and the Lean formalization,
\texttt{dirichlet-revisited/} in the repository
\url{https://gitlab.liu.se/jansn19/arithmetical-functions-truncations}.

\bibitem[Sage]{Sage}
The Sage Developers,
\emph{SageMath, the Sage Mathematics Software System},
\url{https://www.sagemath.org}.

\bibitem[Sne00]{Snellman2000}
Jan Snellman,
\emph{Truncations of the ring of number-theoretic functions},
Homology Homotopy Appl. \textbf{2} (2000), 17--27,
\href{https://doi.org/10.4310/hha.2000.v2.n1.a2}{doi:10.4310/hha.2000.v2.n1.a2};
arXiv:math/9904143.

% NOTE (2026-08-30): the two entries below cite the arXiv preprints ONLY, deliberately.
% They are two distinct deposits -- different identifiers, different submission dates
% (2002-05-23 and 2002-08-23), different titles and abstracts, verified against the deposited
% PDFs.  An earlier version of this bibliography gave BOTH the journal reference "Int. J. Math.
% Game Theory Algebra 13 (2003), no. 6, 485-519", which cannot be right for two separate papers,
% and attributed that page range to arXiv's metadata.  arXiv sets no journal-ref on either
% record; the claim had no source.  zbMATH does have exactly one published item, Zbl 1165.11301,
% at those pages, under a title matching neither preprint and linking only math/0205242, and its
% review treats the general-V case first and V = [n] "in the second part" -- so the two preprints
% were evidently published together as one article.  But Jan checked his office on 2026-08-29 and
% has no reprint of the 2003 paper (he has the IJMGTA 2000 and 2002 ones, both on unrelated
% topics), so which preprint became which part cannot be verified against a copy.  Since the
% preprints are freely accessible and Section 7 of math/0208183 is cited by section number --
% a numbering the merged article would not preserve -- citing the preprints is both the honest
% and the more precise choice.  Labels use the arXiv deposit year, as they do for every other
% preprint in this list (Sne06, Fan23, Hol23, GKP24, Miy26).
\bibitem[Sne02a]{SnellmanUnitaryGeneral}
Jan Snellman,
\emph{The ring of arithmetical functions with unitary convolution: general
truncations},
arXiv:math/0205242.

\bibitem[Sne02b]{SnellmanUnitaryN}
Jan Snellman,
\emph{The ring of arithmetical functions with unitary convolution: the
\([n]\)-truncation},
arXiv:math/0208183.

\bibitem[Sne06]{SnellmanLaplacians}
Jan Snellman,
\emph{Laplacians on shifted multicomplexes},
arXiv:math/0606104.

\bibitem[Ten15]{Tenenbaum}
Gérald Tenenbaum,
\emph{Introduction to Analytic and Probabilistic Number Theory},
3rd ed., Grad. Stud. Math. \textbf{163}, Amer. Math. Soc., 2015.
% NOTE (2026-08-28): Ch. III.6 is "Integers free of small prime factors".  While
% preparing references/CERTIFICATE.md, Fan (arXiv:2306.03347) was found to give the
% locator directly -- its eq. (1.3) is followed by "uniformly for 2 <= y <= sqrt x
% (see [13, Theorem III.6.4])", [13] being this book, and it cites [13, Cor. III.6.5]
% for omega(u) = e^{-gamma} + O(u^{-u/2}).  So the number is III.6.4, not the III.6.7
% an earlier note guessed.  That is SECONDARY corroboration from a paper citing the
% book, not verification against the book, so no theorem number is given here.  Check
% a physical or LiU-subscription copy before adding one.

\end{thebibliography}
\end{document}